\documentclass[11pt]{article}
\usepackage{CJK}
\usepackage{amsfonts}
\usepackage{amsmath}
\usepackage{mathrsfs}
\usepackage{mathrsfs,amscd,amssymb,amsthm,amsmath,url,bm}
\usepackage{color}
\usepackage{mathrsfs,amscd,amssymb,amsthm,amsmath,bm,graphicx,psfrag,subfigure,url}

\makeatletter

\renewcommand{\@seccntformat}[1]{{\csname the#1\endcsname}{\normalsize .}\hspace{.5em}}
\makeatother

\usepackage{indentfirst}

\def \[{\begin{equation}}
\def \]{\end{equation}}

\newtheorem{thm}{Theorem}[section]

\newtheorem{fact}{Fact}
\newtheorem{lem}[thm]{Lemma}
\newtheorem{cor}[thm]{Corollary}

\newtheorem{claim}{Claim}

\newenvironment{wst}
{\setlength{\leftmargini}{1.5\parindent}
 \begin{itemize}
 \setlength{\itemsep}{-1.1mm}}
{\end{itemize}}

\begin{document}
\baselineskip=0.23in
\begin{center}{\Large \bf Matching extension and matching exclusion via the size or the spectral radius of graphs\footnote{Financially supported by the National Natural Science Foundation of China (Grant Nos. 12171190, 11671164)}}
\vspace{4mm}

{\large Shujing Miao$^1$,\ \ Shuchao Li$^1$,\ \ Wei Wei$^{2,}$}\footnote{Corresponding author. \\
\hspace*{5mm}{\it Email addresses}: sjmiao2020@sina.com (S. Miao),\, li@ccnu.edu.cn (S. Li),\, weiweimath@sina.com (W. Wei)} \vspace{2mm}

$^1$Hubei Key Laboratory of Mathematical Science, and Faculty of Mathematics and Statistics,\\  Central China Normal
University, Wuhan 430079, PR China

$^2$Center of Intelligent Computing and Applied Statistics, School of Mathematics, Physics and Statistics, Shanghai University of Engineering Science, Shanghai 201620, China

\end{center}


\noindent {\bf Abstract}: \ A graph $G$ is said to be $k$-extendable if every matching of size $k$ in $G$ can be extended to a perfect matching of $G$, where $k$ is a positive integer. We say $G$ is $1$-excludable if for every edge $e$ of $G$, there exists a perfect matching excluding $e$. In this paper, we first establish a lower bound on the size (resp. the spectral radius) of $G$ to guarantee that $G$ is $k$-extendable. Then we determine a lower bound on the size (resp. the spectral radius) of $G$ to guarantee that $G$ is $1$-excludable. All the corresponding extremal graphs are characterized.

\vspace{2mm} \noindent{\it Keywords}:\ Size; Spectral radius; Matching extension; Marching exclusion

\vspace{2mm}


\setcounter{section}{0}
\section{\normalsize Introduction}\setcounter{equation}{0}

In this paper, we only deal with finite and undirected graphs without loops or multiple edges. Let $G$ be a graph with vertex set $V(G)$ and edge set $E(G).$ The \textit{order} of $G$ is the number $n=| V(G)|$ of its vertices and its \textit{size} is the number $m=|E(G)|$ of its edges. For graph theoretic notation and terminology not defined here, we refer to \cite{C.G01,D.B.}.

The problem of extending a certain substructure of a graph into a spanning subgraph with the same properties was studied extensively. Let $I$ be a set of non-negative integers. A graph $G$ is called an \textit{$I$-graph} if $d_G(v)\in I$ for all $v\in V(G)$. In particular, a $\{r\}$-graph is a $r$-regular graph. An \textit{$I$-factor} is a spanning $I$-subgraph of $G$. Liu \cite{G.Z.} characterized a graph in which each edge can be extended into an $[a, b]$-factor. Let $\mathcal{H}$ be a set of connected graphs and $\mathcal{P}_{\geqslant k}=\{P_i:i\geqslant k\geqslant2\}$. An \textit{$\mathcal{H}$-factor} is a spanning subgraph, whose connected components are isomorphic to graphs from $\mathcal{H}$. Hence, a \textit{$\mathcal{P}_{\geqslant k}$-factor} means a path factor in which each component has order at least $k\geqslant2$.
Zhang and Zhou \cite{Z.Z} described a graph in which each edge can be extended into a $\mathcal{P}_{\geqslant 2}$-factor and $\mathcal{P}_{\geqslant 3}$-factor, respectively.

A set $M$ of edges in a graph is called a \textit{matching} if no two edges of $M$ share a vertex. If a matching covers all the vertices of a graph $G$, then it is called a \textit{perfect matching} of $G$. A matching $M$ is \textit{extendable} in $G$, if there exists a perfect matching of $G$ containing every edge of $M$.  Moreover, a graph with at least $2k+2$ vertices having a perfect matching is said to be $k$-\textit{extendable} if every matching of size $k$ is extendable. A graph $G$ is said to be \textit{$1$-excludable} if, for every edge $e$, $G$ has a perfect matching excluding $e$.

Matching extendability was studied extensively in relation to some important problems in graphs (see the nice survey paper \cite{D-M-2008} and books \cite{A.K,Y.L}). The research on the matching extension seems to be traced back to 1964 when Hetyei \cite{G.H} studied it for bipartite graphs. In 1975, Little, Grant and Holton \cite{C.H.C} gave a necessary and sufficient condition for a graph to be $1$-extendable. Plummer \cite{M.D.} studied the properties of $k$-extendable graphs. In the past years, there has been some focus on distance matching extension of graphs and matching extendability of graphs embedded on surfaces. For more advances along this line, we refer the readers to \cite{AFS2020,AFS2021,APR2020,FS2020} and the references cited therein.

The match exclusion also attracted much attention. Plesn\'{\i}k \cite{P1974} studied $(r-1)$-edge connected $r$-regular multigraph of even order, and determined that for any $r-1$ edges there is a perfect matching excluding these $r-1$ edges. Katerinis \cite{K1993} extended this result and showed that, for any $r-k$ edges, there is a $\{k\}$-factor excluding these $r-k$ edges. Chen and Liu \cite{Cpc} gave the condition that, for any edge of a graph, there is an $[a,b]$-factor containing this edge and another $[a,b]$-factor excluding this edge. Kano and Yu \cite{K2005} studied connected $r$-regular graph, and got a similar condition that, for any edge of a graph, there is a $\{k\}$-factor containing this edge and another $\{k\}$-factor excluding this edge.

Recently, O \cite{S.O.} described how the property of perfect matching interacts with
such other graph parameters as the size and the spectral radius, respectively. He established a sharp upper bound on the size (resp. the spectral radius) of a graph without a perfect matching. The authors \cite{LM} of the current paper studied the relationship between $\mathcal{P}_{\geqslant 2}$-factor and size (resp. spectral radius). For the study of the relationship between certain topological properties of a graph and its spectral radius, one may be referred to those in \cite{A-E,HLZ2022,M.A.,S.C.} and the references therein.

Over the past few years, many works focused on how the property of $k$-extendability interacts with
such other graph parameters as genus, toughness, claw-freedom, degree sums and generalized
neighborhood conditions (see the survey paper \cite{MDP1994}). As far as we know, only a few works in the literature describe the matching extension (or the matching exclusion) by the size, or the spectra of graphs. With this in mind, we establish the relationship between matching extension (resp. exclusion) and size (resp. spectral radius). We obtain some sufficient conditions to guarantee that a graph is $k$-extendable (resp. 1-excludable) via the size or the spectral radius of the graph.

In order to formulate our main results, we introduce several notions.
A graph $G$ is called \textit{trivial} if $|V(G)|=1.$ Let $V_1\subseteq V(G)$ and $E_1\subseteq E(G)$. Then $G-V_1, G-E_1$ are the graphs formed from $G$ by deleting the vertices in $V_1$ and their incident edges, the edges in $E_1,$ respectively. For convenience, denote $G-\{v\}$ and $G-\{uv\}$ by $G-v$ and $G-uv,$ respectively. For a given subset $S\subseteq V(G),$ the subgraph of $G$ induced by $S$ is denoted by $G[S].$

For a vertex $v\in V(G),$ let $N_G(v)$ be the set of all neighbours of $v$ in $G.$ Then $d_G(v) = |N_G(v)|$ is the \textit{degree} of $v$ in $G.$ A vertex $v$ of $G$ is called a \textit{pendant vertex} if $d_G(v)=1.$ The \textit{complement} of a graph $G$ is a graph $\overline{G}$ with the same vertex set as $G,$ in
which any two distinct vertices are adjacent if and only if they are non-adjacent in $G.$
As usual, let $K_n$ denote the complete graph on $n$ vertices. Let $K_{n-1}^+$ denote { the} graph obtained from $K_{n-1}$ by adding a pendant vertex. For two graphs $G_1$ and $G_2,$ we define $G_1\cup G_2$ as their \textit{disjoint union}. The \textit{join} $G_1\vee G_2$ is obtained from $G_1\cup G_2$ by joining every vertex of $G_1$ with every vertex of $G_2$ by an edge.

When we study the relationship between matching exclusion and parameters we only consider the connected graph with minimum degree at least $2$, which is based on the fact: If graph G contains a pendant edge $e$, then there does not exist perfect matching in $G-e$.



Our first main result gives a sufficient condition to ensure {that} a graph $G$ is $k$-extendable according to the size of $G.$
\begin{thm}\label{thm1.1}
Let $k$ be a positive integer and $n$ be an even positive integer. If $G$ is an $n$-vertex connected graph with
$|E(G)|\geqslant$\begin{math}
\left(\begin{smallmatrix}
n-1\\2
\end{smallmatrix} \right)
\end{math}$+2k$, then $G$ is $k$-extendable, unless $G\in\{K_{2k}\vee(K_{n-2k-1}\cup K_1),\ K_{2k+1}\vee3K_1\}$.
\end{thm}
The following result is a direct consequence of Theorem \ref{thm1.1}.
\begin{cor}\label{cor1.2}
Let $n$ be an even positive integer. If $G$ is an $n$-vertex connected graph with \linebreak
$|E(G)|\geqslant$\begin{math}
\left(\begin{smallmatrix}
n-1\\2
\end{smallmatrix} \right)
\end{math}$+2$, then $G$ is $1$-extendable, unless $G\in\{K_2\vee(K_{n-3}\cup K_1),\ K_3\vee3K_1\}$.
\end{cor}

Our next main result gives a sufficient condition to ensure {that} a graph $G$ is $1$-excludable according to the size of $G.$
\begin{thm}\label{thm1.3}
Let $n$ be an even positive integer and $G$ be an $n$-vertex connected graph with $\delta(G)\geqslant2$. Then $G$ is $1$-excludable if the following three conditions hold:
\begin{wst}
\item[{\rm (i)}] $|E(G)|\geqslant10$ for $n=6$, unless $G=K_2\vee(K_2\cup2K_1)$;
\item[{\rm (ii)}] $|E(G)|\geqslant19$ for $n=8$, unless $G=K_3\vee(K_2\cup3K_1)$;
\item[{\rm (iii)}] $|E(G)|\geqslant$\begin{math}
\left(\begin{smallmatrix}
n-2\\2
\end{smallmatrix} \right)
\end{math}$+3$ for $n\geqslant10$, unless $G\in \{K_4\vee(K_{2}^+\cup4K_1),\ K_1\vee(K_2\cup K_{n-3})\}$.
\end{wst}
\end{thm}

Given a graph $G$ of order $n$, the \textit{adjacency matrix} $A(G)=(a_{ij})_{n\times n}$ of $G$ is a $0$-$1$ matrix in which the entry $a_{ij}=1$ if and only if $v_i$ and $v_j$ are adjacent. The eigenvalues of the adjacency matrix $A(G)$ are also called \textit{eigenvalues} of $G$. Note that $A(G)$ is a real non-negative symmetric matrix. Hence, its eigenvalues are real, which can be arranged in nonincreasing order as $\lambda_1(G)\geqslant \cdots \geqslant \lambda_n(G).$ Note that the adjacency spectral radius (or spectral radius, for short) of $G$ is equal to $\lambda_1(G)$, written as $\rho(G)$.

Our third main result gives a sufficient condition to ensure {that} a graph $G$ is $k$-extendable according to the spectral radius of $G.$
\begin{thm}\label{thm1.4}
Let $k$ be a positive integer and $n$ be an even positive integer. If $G$ is an $n$-vertex connected graph with
$\rho(G)\geqslant\rho(K_{2k}\vee(K_{n-2k-1}\cup K_1))$, then $G$ is $k$-extendable, unless $G=K_{2k}\vee(K_{n-2k-1}\cup K_1)$.
\end{thm}
The following result is a direct consequence of Theorem \ref{thm1.4}.
\begin{cor}\label{cor1.5}
Let $n$ be an even positive integer. If $G$ is an $n$-vertex connected graph with
$\rho(G)\geqslant\rho(K_{2}\vee(K_{n-3}\cup K_1))$, then $G$ is $1$-extendable, unless $G=K_{2}\vee(K_{n-3}\cup K_1)$.
\end{cor}
Our last main result gives a sufficient condition to ensure {that} a graph $G$ is $1$-excludable according to the spectral radius of $G.$
\begin{thm}\label{thm1.6}
Let $n$ be an even positive integer and $G$ be an $n$-vertex connected graph with $\delta(G)\geqslant2$. Then $G$ is $1$-excludable if the following three conditions hold:
\begin{wst}
\item[{\rm (i)}] $\rho(G)\geqslant\rho(K_2\vee(K_2\cup2K_1))$ for $n=6$, unless $G=K_2\vee(K_2\cup2K_1)$;
\item[{\rm (ii)}] $\rho(G)\geqslant\rho(K_3\vee(K_2\cup3K_1))$ for $n=8$, unless $G=K_3\vee(K_2\cup3K_1)$;
\item[{\rm (iii)}] $\rho(G)\geqslant\rho(K_1\vee(K_2\cup K_{n-3}))$ for $n\geqslant10$, unless $G=K_1\vee(K_2\cup K_{n-3})$.
\end{wst}
\end{thm}

The proof techniques for our main results follow the idea of O
\cite{S.O.}. Together with some new ideas, we make the proofs work. Our paper is organized as follows. In Section 2, we give some preliminary results. In Section 3, we give the proofs of Theorems \ref{thm1.1} and \ref{thm1.3}. In Section 4, we give the proof of Theorem \ref{thm1.4} and in Section 5, we give the proof of Theorem \ref{thm1.6}. Some concluding remarks are given in the last section.

\section{\normalsize Some preliminaries}
In this section, we present some necessary preliminary results, which play an important role in the subsequent sections. The first one follows directly from \cite[Theorem 6.8]{R.B.}.
\begin{lem}[\cite{R.B.}]\label{lem2.1}
Let $G$ be a connected graph and let $H$ be a proper subgraph of $G$. Then $\rho(G)>\rho(H)$.
\end{lem}

Let $M$ be an $n\times n$ irreducible and nonnegative matrix. Then by the Perron-Frobenius theorem (see \cite{R.A}),  there exists a unit positive eigenvector, say ${\bf x}=(x_1,x_2,\ldots,x_n)^T$, of $M$ corresponding to spectral radius of $M$. As usual, we call ${\bf x}$ the \textit{Perron vector} of $M$.

The following lemma is a direct consequence of \cite[Proposition 16]{0007}.
\begin{lem}[\cite{0007}]\label{lem2.2}
Let $G$ be an $n$-vertex connected graph and let ${\bf x} = (x_1, x_2, \ldots, x_n)^T$ be the Perron vector of $A(G)$ corresponding to $\rho(G)$. If $v_i, \, v_j$ are in $V(G)$ satisfying $N_G(v_i)\setminus \{v_j\} = N_G(v_j)\setminus \{v_i\}$, then $x_i=x_j$.
\end{lem}



The following lemma states that the spectral radius increases after some graph transformation. This lemma is very important in proving our main result.
\begin{lem}[\cite{H.L}]\label{lem2.3}
Let $G$ be a connected graph and let $u$, $v$ be two specified vertices of $G$. Assume that $G^*$ is obtained from $G$ by replacing the edge $vw$ by a new edge $uw$ for some vertices $w$ such that $w\in N_G(v)\setminus N_G(u)$, and ${\bf x}$ is the Perron vector of $A(G)$. If $x_u\geqslant x_v$, then $\rho(G)<\rho(G^*)$.
\end{lem}

Let $M$ be a real matrix whose rows and columns are indexed by $V=\{1, \ldots, n\}$. Assume that $M$, with respect to the partition $\pi:\ V=V_1\cup\cdots\cup V_s$, can be written as
$$
M=\left(
  \begin{array}{ccc}
    M_{11} & \cdots & M_{1s} \\
    \vdots & \ddots & \vdots \\
    M_{s1} & \cdots & M_{ss} \\
  \end{array}
  \right),
$$
where $M_{ij}$ denotes the submatrix (block) of $M$ formed by rows in $V_i$ and columns in $V_j$. Let $q_{ij}$ denote the average row sum of $M_{ij}$. The matrix $M_{\pi}=(q_{ij})$ is called the \textit{quotient matrix} of $M$. If the row sum of each block $M_{ij}$ is a constant, then the partition is \textit{equitable}.
\begin{lem}[\cite{YLH}]\label{lem2.4}
Let $M$ be a square matrix with an equitable partition $\pi$ and let $M_{\pi}$ be the corresponding quotient matrix. Then every eigenvalue of $M_{\pi}$ is an eigenvalue of $M$. Furthermore, if $M$ is nonnegative, then the largest eigenvalues of $M$ and $M_{\pi}$ are equal.
\end{lem}

The subsequent lemma is the well-known Cauchy Interlacing Theorem.
\begin{lem}[\cite{H.M.}]\label{lem2.5}
Let $M$ be a Hermitian matrix of order $s$, and let $N$ be a principal submatrix of $M$ with order $t$. If $\lambda_1\geqslant \lambda_2\geqslant \cdots \geqslant \lambda_s$ are the eigenvalues of $M$ and $\mu_1\geqslant \mu_2\geqslant \cdots \geqslant \mu_t$ are the eigenvalues of $N$, then $\lambda_{i}\geqslant \mu_i\geqslant \lambda_{s-t+i}$ for $1\leqslant i\leqslant t$.
\end{lem}

Let $o(G)$ be the number of odd components (components with odd order) of $G$. In 1995, Chen~\cite{C.P} established a necessary and sufficient condition for a graph to be $k$-extendable, which reads as the following lemma.
\begin{lem}[\cite{C.P}]\label{lem2.6}
Let $k\geqslant1$ be an integer and let $G$ be a simple graph. Then $G$ is $k$-extendable if and only if $o(G-S)\leqslant|S|-2k$ for every $S\subseteq V(G)$ such that $G[S]$ contains $k$ independent edges.
\end{lem}

An edge $e$ of a connected graph $G$ is called an \textit{odd-bridge} if $e$ is a bridge of $G$ and $G-e$ consists of two odd
components. Clearly, such graph $G$ has even order. The next result gives a necessary and sufficient condition for a graph to be $1$-excludable.
\begin{lem}[\cite{C.P.C,C.P}]\label{lem2.7}
Let $G$ be a connected graph. Then $G$ is $1$-excludable if and only if the following two conditions hold for all $S\subseteq V(G):$
\begin{wst}
\item[{\rm (i)}]If $G-S$ has a component containing an odd-bridge, then $o(G-S)\leqslant|S|-2$;
\item[{\rm (ii)}]$o(G-S)\leqslant|S|$, otherwise.
\end{wst}
\end{lem}

The following result established a sufficient condition for a graph to have a perfect matching via its size.
\begin{lem}[\cite{S.O.}]\label{lem2.09}
Let $n\geqslant10$ be an even integer or $n=4$. If $G$ is an $n$-vertex connected graph with $|E(G)|>$\begin{math}
\left(\begin{smallmatrix}
n-2\\2
\end{smallmatrix} \right)
\end{math}$+2$, then $G$ has a perfect matching. For $n=6$ or $n=8$, if $|E(G)|>9$ or $|E(G)|>18$, respectively, then $G$ has a perfect matching.
\end{lem}

The next result established a sufficient condition for a graph to have a perfect matching via its spectral radius.
\begin{lem}[\cite{S.O.}]\label{lem2.010}
Let $n\geqslant8$ be an even integer or $n=4$. If $G$ is an $n$-vertex connected graph with $\rho(G)>\theta(n)$, where $\theta(n)$ is the largest root of $x^3-(n-4)x^2-(n-1)x+2(n-4)=0$, then $G$ has a perfect matching. For $n=6$, if $\rho(G)>\frac{1+\sqrt{33}}{2}$, then $G$ has a perfect matching.
\end{lem}

The next two lemmas establish upper bounds on the size (resp. the spectral radius) of a graph $G$ under the odd components $o(G-S)$ condition, where $S$ is a subset of $V(G)$. They can be obtained in the procedure of the proofs for Lemma \ref{lem2.09} and Lemma \ref{lem2.010}, respectively. For the details, one may be referred to  \cite{S.O.}. 
\begin{lem}\label{lem2.9}
Let $n$ be an even positive integer. If $G$ is an $n$-vertex connected graph and there exists a subset $S\subseteq V(G)$ such that $o(G-S)\geqslant|S|+2$, then
\begin{wst}
\item[{\rm (i)}]$|E(G)|\leqslant9$ for $n=6$;
\item[{\rm (ii)}]$|E(G)|\leqslant18$ for $n=8$;
\item[{\rm (iii)}]$|E(G)|\leqslant$\begin{math}
\left(\begin{smallmatrix}
n-2\\2
\end{smallmatrix} \right)
\end{math}$+2$ for $n\geqslant10$ or $n=4$.
\end{wst}
\end{lem}
\begin{lem}\label{lem2.10}
Let $n$ be an even integer. If $G$ is an $n$-vertex connected graph and there exists a subset $S\subseteq V(G)$ such that $o(G-S)\geqslant|S|+2$, then
\begin{wst}
\item[{\rm (i)}]$\rho(G)\leqslant\frac{1+\sqrt{33}}{2}$ for $n=6$, and $\rho(K_2\vee4K_1)=\frac{1+\sqrt{33}}{2}$;
\item[{\rm (ii)}]$\rho(G)\leqslant\theta(n)$ for $n\geqslant8$ or $n=4$, where $\theta(n)$ is the largest root of $x^3-(n-4)x^2-(n-1)x+2(n-4)=0$, and $\rho(K_1\vee(K_{n-3}\cup2K_1))=\theta(n)$.
\end{wst}
\end{lem}
Next, we compare the size and the spectral radius of $K_s\vee(H\cup K_{n_1}\cup K_{n_2}\cup\ldots\cup K_{n_t})$ ($n_i\geqslant k $ for $i\in\{1,\ldots,t\}$) with those of $K_s\vee(H\cup K_{n-m-s-k(t-1)}\cup (t-1)K_k),$ where $n_1, n_2,\ldots, n_t,s,k$ are positive integers, $H$ is a graph on $m$ vertices, and $n=\sum_{i=1}^tn_i+s+m$.
\begin{lem}\label{lem2.8}
Let $K_s\vee(H\cup K_{n_1}\cup K_{n_2}\cup\cdots\cup K_{n_t})$ and $K_s\vee(H\cup K_{n-m-s-k(t-1)}\cup (t-1)K_k)$ be the graphs defined as above. Assume $n_1=\max\{n_1,n_2,\ldots,n_t\}.$ Then
\[\label{eq:2.001}
  |E(K_s\vee(H\cup K_{n_1}\cup K_{n_2}\cup\cdots\cup K_{n_t}))|\leqslant|E(K_s\vee(H\cup K_{n-m-s-k(t-1)}\cup (t-1)K_k))|
\]
and
\[\label{eq:2.002}
  \rho(K_s\vee(H\cup K_{n_1}\cup K_{n_2}\cup\cdots\cup K_{n_t}))\leqslant\rho(K_s\vee(H\cup K_{n-m-s-k(t-1)}\cup (t-1)K_k)).
\]
Both equalities hold if and only if $n_2=n_3=\cdots= n_t=k$.
\end{lem}
\begin{proof}
Clearly, $K_s\vee(H\cup K_{n_1}\cup K_{n_2}\cup\cdots\cup K_{n_t})\cong K_s\vee(H\cup K_{n-s-m-k(t-1)}\cup (t-1)K_k)$ if $n_2=n_3=\cdots =n_t=k$. Thus, it suffices to show that the strict inequality in \eqref{eq:2.001} (resp. \eqref{eq:2.002}) holds if there exists $i\in\{2,3,\ldots,t\}$ such that $n_i\geqslant k+1$.

Assume that $n_i\geqslant k+1$ for some $i\in \{2,\ldots, t\}$. Then let $$\tilde{G}=K_s\vee(H\cup K_{n_1+1}\cup\cdots\cup K_{n_{i-1}}\cup\cdots\cup K_{n_t}).$$
On the one hand, one has
\begin{align*}
|E(\tilde{G})|=&|E(G)|-(n_i-1)+n_1
        =|E(G)|+(n_1-n_i)+1
        >|E(G)|.
\end{align*}
Thus, $|E(G)|<|E(\tilde{G})|$. This completes the proof of the first part.

On the other hand, assume that ${\bf x}=(x_1,x_2,\ldots,x_n)^T$ is the Perron vector of $A(G)$, and let $x_j$ denote the entry of ${\bf x}$ corresponding to the vertex $v_j\in V (G)$. By Lemma \ref{lem2.2}, one has $x_r=x_s$ for all $v_r, v_s$ in $V(K_s)$ (resp. $V(K_{n_j}), {j\in\{1,2,\ldots,t\}}$). For convenience, let $x_0 =x_r$ for all $v_r\in V(K_s)$, $x_j =x_r$ for all $v_r\in V(K_{n_j})$, $j\in\{1,2,\ldots,t\}$. Then
$$
\left\{
  \begin{array}{ll}
    \rho(G)x_1=sx_0+(n_1-1)x_1, \\[5pt]
    \rho(G)x_i=sx_0+(n_i-1)x_i.
  \end{array}
\right.
$$
Thus, $x_1\geqslant x_i$. By the Rayleigh quotient, we have
\begin{align*}
\rho(\tilde{G})-\rho(G)\geqslant& {\bf x}^T(A(\tilde{G})-A(G)){\bf x}
                 =2n_1x_1x_i-2(n_i-1)x_i^2
                 \geqslant 2x_i^2(n_1-n_i+1)
                 >0.
\end{align*}
Therefore, $\rho(G)<\rho(\tilde{G})$. This completes the proof of the second part.
\end{proof}

Let $K_p+K_q$ denote the graph obtained from $K_p\cup K_q$ by adding an edge connecting one vertex in $K_p$ and one in $K_q$. We close this section by showing the following lemma.

\begin{lem}\label{lem2.11}Given an even positive integer $l$, let $\mathcal{G}=\{K_p+K_q: p+q=l,\  p,\ q$ are odd and positive$\}$.
\begin{wst}
\item[{\rm (i)}]For all $G\in\mathcal{G}$, one has $|E(G)|\leqslant|E(K_{l-1}^+)|$ and $\rho(G)\leqslant\rho(K_{l-1}^+)$, both equalities hold if and only if $G=K_{l-1}^+$;
\item[{\rm (ii)}]For all $G\in\mathcal{G}\setminus\{K_{l-1}^+\}$, one has $|E(G)|\leqslant|E(K_3+K_{l-3})|$ and $\rho(G)\leqslant\rho(K_3+K_{l-3})$, both equalities hold if and only if $G=K_3+K_{l-3}$.
\end{wst}
\end{lem}

\begin{proof}\ \
(i)\  Assume that $G\in\mathcal{G}$. It is clear that $G=K_{l-1}^+$ if $p=1$ or $q=1$. Next, we are to prove $|E(G)|<|E(K_{l-1}^+)|$ and $\rho(G)<\rho(K_{l-1}^+)$ for $p,\ q\geqslant3$.

On the one hand,
\begin{align*}
|E(K_{l-1}^+)|=&|E(G)|+p(q-1)-(q-1)
              =|E(G)|+(p-1)(q-1)
              >|E(G)|,
\end{align*}
as desired.

On the other hand, let $u$ and $v$ denote the two vertices in $G$ whose degrees are $q$ and $p$, respectively. Let ${\bf x}$ be the Perron vector of $A(G)$. Without loss of generality, we assume $x_u\geqslant x_v$. Then we construct a new graph $G'$ as
$$
G'=G-\sum_{w\in N_G(v)\setminus\{u\}}wv+\sum_{w\in N_G(v)\setminus\{u\}}uw.
$$
By Lemma \ref{lem2.3}, we have $\rho(G)<\rho(G')$. Obviously, $G'$ is a proper subgraph of $K_{l-1}^+$, by Lemma \ref{lem2.1}, we have $\rho(G)<\rho(K_{l-1}^+)$, as desired.

(ii) \  Assume $G\in\mathcal{G}\setminus\{K_{l-1}^+\}$, then $p,\ q\geqslant3$. Clearly, $G=K_3+K_{l-3}$ if $p=3$ or $q=3$. In what follows, we shall show $|E(G)|<|E(K_3+K_{l-3})|$ and $\rho(G)<\rho(K_3+K_{l-3})$ for $p,\ q\geqslant5$.

On the one hand,
\begin{align*}
|E(K_3+K_{l-3})|=&|E(G)|+p(q-3)-3(q-3)
              =|E(G)|+(p-3)(q-3)
              >|E(G)|,
\end{align*}
as desired.

On the other hand, let $u$ and $v$ denote the two vertices in $G$ whose vertex degrees are $q$ and $p$, respectively. The quotient matrix of $A(G)$ corresponding to the partition $V(G)=V(K_{q-1})\cup \{u\}\cup \{v\}\cup V(K_{p-1})$ is given as
$$
A=\left(
  \begin{array}{cccc}
   q-2 & 1 & 0 & 0 \\
   q-1 & 0 & 1 & 0 \\
    0 & 1 & 0 & l-q-1\\
    0 & 0 & 1 & l-q-2\\
  \end{array}
\right),
$$
whose characteristic polynomial is
\[\label{eq:2.01}
\Phi_1(x)=x^4+(4-l)x^3+(lq-q^2-3l+5)x^2+(2lq-2q^2-2l)x+l-3.
\]
Substituting $q=3$ into \eqref{eq:2.01} gives the characteristic polynomial of the quotient matrix of $K_3+K_{l-3},$ via the partition $V(K_3+K_{l-3})=V(K_2)\cup \{u\}\cup \{v\}\cup V(K_{{l-4}})$, as
\[\label{eq:2.1}
\Phi_2(x)=x^4+(4-l)x^3-4x^2+(4l-18)x+l-3.
\]
Clearly, the vertex partition is equitable. Hence, by Lemma \ref{lem2.4}, the largest roots of $\Phi_1(x)=0$ and $\Phi_2(x)=0$ equal $\rho(G)$ and $\rho(K_3+K_{l-3})$, respectively.
Plugging the value $\rho(G)$ into $x$ of $\Phi_2(x)-\Phi_1(x)$ gives us
$$
\Phi_2(\rho(G))-\Phi_1(\rho(G))=-\rho(G)(q-3)(l-q-3)(\rho(G)+2)<0.
$$
Note that $\Phi_1(\rho(G))=0$. Thus $\Phi_2(\rho(G))<0$ and so $\rho(G)<\rho(K_3+K_{l-3})$ for $p,\ q\geqslant5$.

This completes the proofs of (i) and (ii).
\end{proof}
\section{\normalsize Proof of Theorems \ref{thm1.1} and \ref{thm1.3}}\setcounter{equation}{0}
In this section, we give the proof of Theorems \ref{thm1.1} and \ref{thm1.3}, in which the former gives a sufficient condition via the size of a connected graph to ensure that the graph is $k$-extendable, the latter presents a sufficient condition via size to determine that a graph is $1$-excludable.

\begin{proof}[\bf Proof of Theorem \ref{thm1.1}]
By Lemma \ref{lem2.09}, we obtain that $G$ has a perfect matching. Thus, $G$ must contain $k$ independent edges. Suppose to the contrary that $G$ is not $k$-extendable. By Lemma \ref{lem2.6}, there exists a vertex subset $S\subseteq V(G)$ such that $G[S]$ contains $k$ independent edges and $o(G-S)>|S|-2k$. Since $n$ is even, we have $o(G-S)\equiv|S|\pmod 2$. Thus $o(G-S)\geqslant|S|-2k+2$ and $|S|\geqslant2k$. Choose such a connected graph $G$ such that its size is as large as possible.

According to the choice of $G$, the induced subgraph $G[S]$ and each connected component of $G-S$ are complete graphs, respectively. Furthermore, all components of $G-S$ are odd.

For convenience, let $o(G-S)=q$ and $|S|=s$, then $n\geqslant2s-2k+2$. Assume that $G_1,G_2,\ldots,G_q$ are all the components of $G-S$ and let {$n_i=|V(G_i)|$ for $i=1,\ldots,q.$ For convenience, let $n_1\geqslant n_2\geqslant\cdots\geqslant n_q\geqslant1$}. Then, $G=K_s\vee(K_{n_1}\cup K_{n_2}\cup\cdots\cup K_{n_q})$. By Lemma \ref{lem2.8}, we have $|E(G)|\leqslant|E(K_s\vee(K_{n-s-q+1}\cup (q-1)K_1))|$ with equality if and only if $G=K_s\vee(K_{n-s-q+1}\cup (q-1)K_1)$. Note that $o(K_s\vee(K_{n-s-q+1}\cup (q-1)K_1)-V(K_s))=o(G-S)\geqslant s-2k+2$. According to the choice of $G$, we obtain $G=K_s\vee(K_{n-s-q+1}\cup (q-1)K_1)$.

Notice that $q\geqslant s-2k+2$. It is clear that $K_s\vee(K_{n-s-q+1}\cup (q-1)K_1)$ is a subgraph of $K_s\vee(K_{n-2s+2k-1}\cup(s-2k+1)K_1)$ and
$o(K_s\vee(K_{n-2s+2k-1}\cup(s-2k+1)K_1)-V(K_s))=s-2k+2$. According to the choice of $G$, we have $G=K_s\vee(K_{n-2s+2k-1}\cup(s-2k+1)K_1)$. Thus $|E(G)|=s(s-2k+1)+$ \begin{math}
\left(\begin{smallmatrix}
n-s+2k-1\\2
\end{smallmatrix} \right)
\end{math}. In what follows, we are to prove $|E(G)|\leqslant2k+$\begin{math}
\left(\begin{smallmatrix}
n-1\\2
\end{smallmatrix} \right)
\end{math}. By a direct computation, we obtain
\begin{align*}
\dbinom{n-1}{2}+2k-|E(G)|=&\dbinom{n-1}{2}+2k-\dbinom{n-s+2k-1}{2}-s(s-2k+1)\\
                         =&\frac{(s-2k)(2k+2n-3s-5)}{2}.
\end{align*}
Obviously, $|E(G)|=2k+$\begin{math}
\left(\begin{smallmatrix}
n-1\\2
\end{smallmatrix} \right)
\end{math} if $s=2k$. So it suffices to prove $2k+2n-3s-5\geqslant0$ for $s\geqslant2k+1$.
Recall that $n\geqslant2s-2k+2$, we have
$$
2n+2k-3s-5\geqslant s-2k-1\geqslant0.
$$
Note that $2n+2k-3s-5=0$ if and only if $s=2k+1$ and $n=2k+4$.
Thus, $|E(G)|\leqslant2k+$\begin{math}
\left(\begin{smallmatrix}
n-1\\2
\end{smallmatrix} \right)
\end{math} with equality if and only if $s=2k$ or $s=2k+1$ and $n=2k+4$ (i.e., $G=K_{2k}\vee(K_{n-2k-1}\cup K_1)$ or $G=K_{2k+1}\vee3K_1$), a contradiction.

This completes the proof.
\end{proof}

\begin{proof}[\bf Proof of Theorem \ref{thm1.3}]
Assume that $e$ is an edge of $G$. Suppose to the contrary that $G$ has no perfect matching excluding $e$. By Lemma \ref{lem2.7}, there exists a subset $S\subseteq V(G)$ such that the following statements {hold}:
\begin{wst}
\item[{\rm (1)}] If $G-S$ has a component containing an odd-bridge, then $o(G-S)>|S|-2$;
\item[{\rm (2)}] $o(G-S)>|S|$, otherwise.
\end{wst}

Choose such a connected graph $G$ so that its size is as large as possible. We proceed by considering {the} following two possible cases.

{\bf Case 1.} $G-S$ has no component containing an odd-bridge.

Note that $n$ is even and $o(G-S)\equiv|S|\pmod 2$. Thus, $o(G-S)\geqslant|S|+2$. By {Lemma} \ref{lem2.9}, we may easily get contradictions.

{\bf Case 2.} $G-S$ has a component containing an odd-bridge $e=uv$, say $\hat{G_1}$.

Note that $o(G-S)\equiv|S|\pmod 2$. Thus, $o(G-S)\geqslant|S|$.
Firstly, we assume $S=\emptyset$. Then $o(G)=o(G-S)\geqslant0$. Since $G$ is connected and $n$ is even, we have $o(G)=0$, and $G=\hat{G_1}$. Let $n_u$ and $n_v$ denote the number of vertices of connected components containing vertices $u$ and $v$ in $G-e$, respectively. Since $\delta(G)\geqslant2$, we have $n_u,\ n_v\geqslant3$. According to the choice of $G$, we may obtain $G=K_{n_u}+K_{n_v}$. Next, we proceed by showing the following fact.
\begin{fact}\label{fact1}
Assume $G$ is the graph described above.
\begin{wst}
\item[\rm (a)] If $n=6$, then $|E(G)|<10$;
\item[\rm (b)] If $n=8$, then $|E(G)|<19$;
\item[\rm (c)] If $n\geqslant10$, then $|E(G)|<$\begin{math}
\left(\begin{smallmatrix}
n-2\\2
\end{smallmatrix} \right)
\end{math}$+3$.
\end{wst}
\end{fact}
\begin{proof}[\bf Proof of Fact \ref{fact1}] Assume $G=K_{n_u}+K_{n_v}$, and $n_u,\ n_v\geqslant3$ are odd.

(a) If $n=6$, together with Lemma \ref{lem2.11}, we know that $|E(G)|\leqslant|E(K_3+K_{3})|=7<10$, as desired.

(b) If $n=8$, by Lemma \ref{lem2.11}, we also get $|E(G)|\leqslant|E(K_3+K_{5})|=14<19$, as desired.

(c) For $n\geqslant10$, we compare $|E(K_3+K_{n-3})|$ with \begin{math}
\left(\begin{smallmatrix}
n-2\\2
\end{smallmatrix} \right)
\end{math}$+3$. Through some computations, we have
$$
|E(K_3+K_{n-3})|-\dbinom{n-2}{2}-3=4-n<0.
$$
Therefore, by Lemma \ref{lem2.11}, we have $|E(G)|\leqslant|E(K_3+K_{n-3})|<$\begin{math}
\left(\begin{smallmatrix}
n-2\\2
\end{smallmatrix} \right)
\end{math}$+3$, as desired.
\end{proof}
By Fact \ref{fact1}, we get contradictions. Now, we assume $S\neq \emptyset$. Then $o(G-S)\geqslant|S|\geqslant1$.
According to the choice of $G,$ one has
\begin{wst}
\item $G-S$ contains only one even connected component, i.e., $\hat{G_1}$;
\item $G[S]$ (resp. each odd connected component of $G-S$) is a complete graph;
\item $G$ is the join of $G[S]$ and $G-S$, i.e., $G=G[S]\vee (G-S)$.
\end{wst}

For convenience, put $o(G-S)=q$, $|V(\hat{G_1})|=\hat{n}$ and $|S|=s$. Then $n\geqslant s+q+\hat{n}\geqslant2s+2$.
Assume that $G_1,G_2,\ldots,G_q$ are all the odd components of $G-S$ and let {$n_i=|V(G_i)|$ for $i=1,\ldots,q.$ For convenience, let $n_1\geqslant n_2\geqslant\cdots\geqslant n_q$}. Then $n=s+\hat{n}+\sum_{i=1}^{q}n_i$ and $G=K_s\vee(\hat{G_1}\cup K_{n_1}\cup K_{n_2}\cup\ldots\cup K_{n_q})$.

Recall that $\hat{G_1}$ is the unique even component of $G-S$ containing an odd-bridge $e=uv$. Assume that $G_u$ and $G_v$ are the two odd components of $\hat{G_1}-e$ containing vertices $u$ and $v$, respectively. By the choice of $G$, we see that $G_u$ and $G_v$ are both complete graphs. Let $|V(G_u)|=n_u$ and $|V(G_v)|=n_v$. Then $\hat{G_1}=K_{n_v}+K_{n_u}$ and $\hat{n}=n_u+n_v$. Without loss of generality, we assume $n_u\geqslant n_v$. We are to show $n_v=1$.
Suppose that $n_v\geqslant3$.

Construct a new graph $\hat{H}=K_{n_u+2}+K_{n_v-2}$, and let $H_1=K_s\vee(\hat{H}\cup K_{n_1}\cup K_{n_2}\cup\ldots\cup K_{n_q})$. We see that
\begin{align*}
|E(H_1)|=&|E(G)|-2(n_v-2)+2n_u=|E(G)|+2(n_u-n_v)+4\geqslant |E(G)|+4.
\end{align*}
Clearly, $|E(H_1)|>|E(G)|$, contradicting the choice of $G$.
Thus, $n_v=1$, and so $\hat{G_1}=K_{\hat{n}-1}^+$. Then $G=K_s\vee(K_{\hat{n}-1}^+\cup K_{n_1}\cup K_{n_2}\cup\ldots\cup K_{n_q}).$

For $s=1$, one has $n_q\geqslant3$. By Lemma \ref{lem2.8}, we have $|E(G)|\leqslant |E(K_1\vee(K_{\hat{n}-1}^+\cup K_{n-\hat{n}-1-3(q-1)}\cup (q-1)K_3)|$ with equality if and only if $G=K_1\vee(K_{\hat{n}-1}^+\cup K_{n-\hat{n}-1-3(q-1)}\cup (q-1)K_3)$.

Recall that $q\geqslant s$. It is clear that $K_1\vee(K_{\hat{n}-1}^+\cup K_{n-\hat{n}-1-3(q-1)}\cup (q-1)K_3)$ is a subgraph of $K_1\vee(K_{\hat{n}-1}^+\cup K_{n-\hat{n}-1})$ and $K_1\vee(K_{\hat{n}-1}^+\cup K_{n-\hat{n}-1})$ satisfies (1). According to the choice of $G$, one has $G=K_1\vee(K_{\hat{n}-1}^+\cup K_{n-\hat{n}-1})$. Next, we show the following fact to get contradictions.
\begin{fact}\label{fact2} Assume $G$ is the graph described above.
\begin{wst}
\item[\rm (a)] If $n=6$, then $|E(G)|<10$;
\item[\rm (b)] If $n=8$, then $|E(G)|<19$;
\item[\rm (c)] If $n\geqslant10$, then $|E(G)|\leqslant$\begin{math}
\left(\begin{smallmatrix}
n-2\\2
\end{smallmatrix} \right)
\end{math}$+3$ with equality if and only if $G=K_1\vee(K_2\cup K_{n-3})$.
\end{wst}
\end{fact}
\begin{proof}[\bf Proof of Fact \ref{fact2}]\ Note that $n-\hat{n}-1\geqslant3$. Then $2\leqslant\hat{n}\leqslant n-4$.

(a) If $n=6$, then $\hat{n}=2$ and so $G=K_1\vee(K_2\cup K_3)$. By a simple calculation, we have $|E(K_1\vee(K_2\cup K_3))|=9<10$, as desired.

(b) If $n=8$, then $\hat{n}=2,4$ and so $G\in\{K_1\vee(K_2\cup K_5), K_1\vee(K_3^+\cup K_3)\}$. By some direct calculations, we have $|E(K_1\vee(K_2\cup K_5))|=18$ and $|E(K_1\vee(K_3^+\cup K_3))|=14$, as desired.

(c) We compare $|E(G)|$ and \begin{math}
\left(\begin{smallmatrix}
n-2\\2
\end{smallmatrix} \right)
\end{math}$+3$. Then
$$
\dbinom{n-2}{2}+3-|E(G)|=(\hat{n}-2)(n-\hat{n}-2).
$$
Since $2\leqslant\hat{n}\leqslant n-4$, one has $(\hat{n}-2)(n-\hat{n}-2)\geqslant0$ with equality if and only if $\hat{n}=2$. Therefore $|E(G)|$ and \begin{math}
\left(\begin{smallmatrix}
n-2\\2
\end{smallmatrix} \right)
\end{math}$+3$ with equality if and only if $G=K_1\vee(K_2\cup K_{n-3})$.
\end{proof}

By Fact \ref{fact2}, we get contradictions. Next, we consider $s\geqslant2$. Then $n_q\geqslant1$.
By Lemma~\ref{lem2.8}, we have $|E(G)|\leqslant |E(K_s\vee(K_{\hat{n}-1}^+\cup K_{n-s-\hat{n}-q+1}\cup(q-1)K_1))|$ with equality if and only if $G=K_s\vee(K_{\hat{n}-1}^+\cup K_{n-s-\hat{n}-q+1}\cup(q-1)K_1)$. Notice that $o(K_s\vee(K_{\hat{n}-1}^+\cup K_{n-s-\hat{n}-q+1}\cup(q-1)K_1)-V(K_s))=o(G-S)\geqslant s$. According to the choice of $G$, we have $G=K_s\vee(K_{\hat{n}-1}^+\cup K_{n-s-\hat{n}-q+1}\cup(q-1)K_1)$.

Recall that $q\geqslant s$. It is clear that $K_s\vee(K_{\hat{n}-1}^+\cup K_{n-s-\hat{n}-q+1}\cup(q-1)K_1)$ is a subgraph of $K_s\vee(K_{\hat{n}-1}^+\cup K_{n-2s-\hat{n}+1}\cup(s-1)K_1)$. According to the choice of $G$, one has $G=K_s\vee(K_{\hat{n}-1}^+\cup K_{n-2s-\hat{n}+1}\cup(s-1)K_1)$.

Put $t=n-2s-\hat{n}+1$. In what follows, we prove $t=1,$ or $\hat{n}=2$. Suppose that $t>1$ and $\hat{n}>2$. Consider a new graph $H_2=K_s\vee (K_{n-2s-1}^+\cup sK_1).$ Clearly, $H_2-V(K_s)$ has a component $K_{n-2s-1}^+$ containing an odd-bridge and $o(H_2-V(K_s))=s>|V(K_s)|-2$. We find that
\begin{align*}
|E(H_2)|=|E(G)|-(t-1)+(\hat{n}-1)(t-1)
        =|E(G)|+(t-1)(\hat{n}-2)
        >|E(G)|,
\end{align*}
 i.e., $|E(H_2)|>|E(G)|$, a contradiction. Thus, $t=1$ or $\hat{n}=2$. So we have $$|E(G)|\leqslant\max\{|E(K_s\vee(K_{n-2s-1}^{+} \cup sK_1))|,\ \ |E(K_s\vee(K_2\cup K_{n-2s-1}\cup(s-1)K_1))|\}.$$

Next, we proceed by showing the following fact.
\begin{fact}\label{fact3}
Let $G=K_s\vee(K_{n-2s-1}^{+} \cup sK_1)$ or $K_s\vee(K_2\cup K_{n-2s-1}\cup(s-1)K_1)$.
\begin{wst}
\item[\rm (a)] If $n=6$, then $E(G)=10$ and $G=K_2\vee(K_2\cup2K_1)$;
\item[\rm (b)] If $n=8$, then $E(G)\leqslant19$ with equality if and only if $G=K_3\vee(K_2\cup3K_1)$;
\item[\rm (c)] If $n\geqslant10$, then $E(G)\leqslant$\begin{math}
\left(\begin{smallmatrix}
n-2\\2
\end{smallmatrix} \right)
\end{math}$+3$ with equality if and only if $G=K_4\vee(K_{2}\cup4K_1)$.
\end{wst}
\end{fact}

\begin{proof}[\bf Proof of Fact \ref{fact3}]\ Recall $n\geqslant2s+2$ and $s\geqslant2$. Thus $2\leqslant s\leqslant\frac{n-2}{2}$.

(a) If $n=6$, then $s=2$ and $G=K_2\vee(K_2\cup2K_1).$ By a direct calculation, we have $|E(K_2\vee(K_2\cup2K_1))|=10$, as desired.

(b) If $n=8$, then $s=2,3$ and $G\in\{K_2\vee(K_3^+\cup2K_1), K_2\vee(K_2\cup K_3\cup K_1), K_3\vee(K_2\cup3K_1)\}$. By some computations, we have $|E(K_2\vee(K_3^+\cup2K_1))|=|E(K_2\vee(K_2\cup K_3\cup K_1))|=17$ and $|E(K_3\vee(K_2\cup3K_1))|=19$, as desired.

(c) Obviously, $|E(K_s\vee(K_{n-2s-1}^{+} \cup sK_1))|=|E(K_s\vee(K_2\cup K_{n-2s-1}\cup(s-1)K_1))|=s(s+1)+1+$\begin{math}
\left(\begin{smallmatrix}
n-s-1\\2
\end{smallmatrix} \right)
\end{math}.
We compare $|E(G)|$ and \begin{math}
\left(\begin{smallmatrix}
n-2\\2
\end{smallmatrix} \right)
\end{math}$+3$. Then
\begin{align}\label{eq:4.1}
\dbinom{n-2}{2}+3-|E(G)|=&\dbinom{n-2}{2}+3-\dbinom{n-s-1}{2}-s(s+1)-1\notag\\
                        =&\frac{(s-1)(2n-3s-8)}{2}.
\end{align}

We first consider $n=2s+2$. As $n\geqslant10$, we have $s\geqslant4$. Clearly, $s-1>0$ and $2n-3s-8=s-4\geqslant0$, with equality if and only if $s=4$. By \eqref{eq:4.1}, we have $E(G)\leqslant$
\begin{math}
\left(\begin{smallmatrix}
n-2\\2
\end{smallmatrix} \right)
\end{math}$+3$, with equality if and only if $G=K_4\vee(K_{2}\cup4K_1)$.

Next, we assume $n\geqslant2s+4$.
Since $s\geqslant2$, one has $2n-3s-8\geqslant s>0$. By \eqref{eq:4.1}, {one has} $|E(G)|<$\begin{math}
\left(\begin{smallmatrix}
n-2\\2
\end{smallmatrix} \right)
\end{math}$+3$.

Therefore, for $n\geqslant10$, one has $|E(G)|\leqslant$\begin{math}
\left(\begin{smallmatrix}
n-2\\2
\end{smallmatrix} \right)
\end{math}$+3$, with equality if and only if $G=K_4\vee(K_{2}\cup4K_1)$.

we complete the proof of Fact \ref{fact2}.
\end{proof}
By Fact \ref{fact2}, we may obtain contradictions.
Together with Cases 1 and 2, our result holds.
\end{proof}

\section{\normalsize Proof of Theorem \ref{thm1.4}}\setcounter{equation}{0}
In this section, we give the proof of Theorem \ref{thm1.4}, which gives a sufficient condition via the spectral radius of a connected graph to ensure that the graph is $k$-extendable.
\begin{proof}[\bf Proof of Theorem \ref{thm1.4}]
By Lemmas \ref{lem2.010} and \ref{lem2.10}, we obtain that $G$ has a perfect matching. Thus $G$ must contain $k$ independent edges. Suppose to the contrary that $G$ is not $k$-extendable. By Lemma \ref{lem2.6}, there exists a vertex subset $S\subseteq V(G)$ satisfying $G[S]$ contains $k$ independent edges and $o(G-S)>|S|-2k$. Since $n$ is even, we have $o(G-S)\equiv|S|\pmod 2$. Thus $o(G-S)\geqslant|S|-2k+2$ and $|S|\geqslant2k$. Choose such a connected graph such that its spectral radius is as large as possible.

Together with Lemma \ref{lem2.1} and the choice of $G$, the induced subgraph $G[S]$ and each connected component of $G-S$ are complete graphs, respectively. Furthermore, all components of $G-S$ are odd.

For convenience, let $o(G-S)=q$ and $|S|=s$. Then $n\geqslant2s-2k+2$. Assume that $G_1,G_2,\ldots,G_q$ are all the components of $G-S$. For convenience, let $n_i=|V(G_i)|$ for $i=1,2,\ldots,q$ and assume $n_1\geqslant n_2\geqslant\cdots\geqslant n_q\geqslant1$. Then $G=K_s\vee(K_{n_1}\cup K_{n_2}\cup\cdots\cup K_{n_q})$. By Lemma \ref{lem2.8}, we have $\rho(G)\leqslant\rho(K_s\vee(K_{n-s-q+1}\cup (q-1)K_1))$ with equality if and only if $G=K_s\vee(K_{n-s-q+1}\cup (q-1)K_1)$. Note that $o(K_s\vee(K_{n-s-q+1}\cup (q-1)K_1)-V(K_s))=o(G-S)\geqslant s-2k+2$. According to the choice of $G$, we obtain that $G=K_s\vee(K_{n-s-q+1}\cup (q-1)K_1)$.

Notice that $q\geqslant s-2k+2$. It is clear that $K_s\vee(K_{n-s-q+1}\cup (q-1)K_1)$ is a subgraph of $K_s\vee(K_{n-2s+2k-1}\cup(s-2k+1)K_1)$ and
$o(K_s\vee(K_{n-2s+2k-1}\cup(s-2k+1)K_1)-V(K_s))=s-2k+2$. Together with Lemma \ref{lem2.1} and the choice of $G$, one has $G=K_s\vee(K_{n-2s+2k-1}\cup(s-2k+1)K_1)$.

In what follows, we show that $\rho(G)\leqslant\rho(H)$ with equality if and only if $G=H$, where $H=K_{2k}\vee(K_{n-2k-1}\cup K_1)$.

Clearly, $G=H$ if $s=2k$. Therefore, it suffices to show that $\rho(G)<\rho(H)$ for $s\geqslant2k+1$.
According to the partition $V(H)=V(K_{2k})\cup V(K_{n-2k-1})\cup V(K_1)$, the quotient matrix of $A(H)$ is
$$
B_H=
\left(
  \begin{array}{ccc}
  2k-1 & n-2k-1 &1\\
  2k &n-2k-2 & 0\\
  2k &0 & 0\\
  \end{array}
\right).
$$
It is routine to check that the characteristic polynomial of $B_H$ is
$$
\Phi(B_H,x)=x^3-(n-3)x^2-(2k+n-2)x-4k^2+2kn-4k.
$$
Obviously, the vertex partition of $H$ is equitable, by Lemma \ref{lem2.4}, the largest root, say $\rho_H$, of $\Phi(B_H,x)=0$ equals $\rho(H)$.

Firstly, we consider $s=2k+1$, then $n\geqslant2k+4$. If $n=2k+4$, then $G=K_{2k+1}\vee3K_1$ and $H=K_{2k}\vee(K_3\cup K_1)$. According to the partition $V(G)=V(K_{2k+1})\cup V(3K_1)$, the quotient matrix of $A(G)$ is
$$
M_{G}=
\left(
  \begin{array}{ccc}
    2k & 3 \\
    2k+1 & 0\\
  \end{array}
\right).
$$
By a simple calculation, one has that the characteristic polynomial of $M_G$ is
$$
\Phi(M_G,x)=x^2-2kx-6k-3.
$$
Clearly, the vertex partition of $G$ is equitable. By Lemma \ref{lem2.4}, the largest root $\rho_1=k+\sqrt{k^2+6k+3}$ of $\Phi(M_G,x)=0$ equals $\rho(G)$.
By plugging the value $\rho_1$ into $x$ of $\Phi(B_H,x)-x\Phi(M_G,x)$, we have
\begin{align*}
\Phi(B_H,\rho_1)-\rho_1\Phi(M_G,\rho_1)=&-\rho_1^2+2k\rho_1+\rho_1+4k
                                                  =-(k+3)+\sqrt{k^2+6k+3}<0.
\end{align*}
The inequality follows by the fact that $(k+3)^2-(k^2+6k+3)=6>0$. Since $\Phi(M_G,\rho_1)=0$, we have $\Phi(B_H,\rho_1)=\Phi(B_H,\rho_1)-\rho_1\Phi(M_G,\rho_1)<0$. Thus, $\rho(G)<\rho(H)$ for $s=2k+1$ and $n=2k+4$.

Next, we consider the case for $s=2k+1$ and $n\geqslant 2k+6$. In this case, the quotient matrix of $A(G)$ corresponding to the partition $V(G)=V(K_s)\cup V(K_{n-2s+2k-1})\cup V((s-2k+1)K_1)$ is
$$
B_G=
\left(
  \begin{array}{ccc}
    s-1 & n-2s+2k-1 & s-2k+1\\
    s & n-2s+2k-2& 0\\
    s & 0 & 0\\
  \end{array}
\right).
$$
By a simple calculation, we obtain that the characteristic polynomial of $B_G$ is
\[\label{eq:4.01}
\Phi(B_G,x)=x^3-(n-s+2k-3)x^2-(n-2sk+s^2+2k-2)x-s(2k-s-1)(2k+n-2s-2).
\]
Obviously, the partition $V(G)=V(K_s)\cup V(K_{n-2s+2k-1})\cup V((s-2k+1)K_1)$ is equitable. By Lemma~\ref{lem2.4}, the largest root, say $\rho$, of $\Phi(B_G,x)=0$ equals $\rho(G)$.

Note that $s=2k+1$. Hence, $G=K_{2k+1}\vee(K_{n-2k-3}\cup2K_1)$. By (\ref{eq:4.01}), we have that
$$
\Phi(B_G,x)=x^3-(n-4)x^2-(4k-1+n)x-8k^2+4kn-20k+2n-8.
$$
By plugging the value $\rho$ into $x$ of $\Phi(B_H,x)-\Phi(B_G,x)$, we have
$$\Phi(B_H,\rho)-\Phi(B_G,\rho)=-\rho^2+(2k+1)\rho+4k^2-2kn+16k-2n+8.$$
Let $\Psi_1(x)=-x^2+(2k+1)x+4k^2-2kn+16k-2n+8$ be a real function in $x$, where $x\in[n-3,+\infty)$. By a direct calculation,
$$\Psi_1'(x)=-2x+2k+1, \ \ \ \Psi_1''(x)=-2<0.$$
Hence, $\Psi_1'(x)$ is a decreasing function. Consequently, $\Psi_1'(x)\leqslant\Psi_1'(n-3)=7+2k-2n\leqslant-5-2k<0$ for $x\geqslant n-3$ and $n\geqslant2k+6$. That is to say $\Psi_1(x)$ is a decreasing on $[n-3,+\infty)$. Note that $K_{n-2}$ is a proper subgraph of $G=K_{2k+1}\vee(K_{n-2k-3}\cup2K_1)$. By Lemma \ref{lem2.1}, one has $\rho(G)>\rho(K_{n-2})=n-3$. Therefore, $\Psi_1(\rho)<\Psi_1(n-3)$. By a simple computation, we get
\begin{align*}
\Psi_1(n-3)=&-n^2+5n+4k^2+10k-4\\
           \leqslant&-(2k+6)^2+5(2k+6)+4k^2+10k-4\tag{As $\frac{5}{2}<2k+6\leqslant n$}\\
           =&-4k-10<0.
\end{align*}
Note that $\Phi(B_G,\rho)=0$, we have $\Phi(B_H,\rho)=\Psi_1(\rho)<\Psi_1(n-3)<0$. Thus, $\rho(G)<\rho(H)$ for $s=2k+1$ and $n\geqslant2k+6$.

Now we consider $s\geqslant2k+2$. Let $\lambda_1=\rho(G)\geqslant\lambda_2\geqslant\lambda_3$ be the three roots of $\Phi(B_G,x)=0$. In what follows, we shall prove that $\lambda_2<\rho_H$ and $\Phi(B_G,\rho_H)>0$.

Let $D={\rm diag}(s,n-2s+2k-1,s-2k+1)$. It is easy to check that $B=D^{\frac{1}{2}}B_GD^{-\frac{1}{2}}$ is symmetric, and also contains
$$
B_2=\left(
  \begin{array}{cc}
    n-2s+2k-2 & 0 \\
    0 & 0 \\
  \end{array}
\right)
$$
as its submatrix. Since $B$ and $B_G$ have the same eigenvalues, by the Cauchy {Interlacing Theorem}, we have $\lambda_2(B)=\lambda_2\leqslant\lambda_1(B_2)=n-2s+2k-2<n-2$. Note that $K_{n-1}$ is a proper subgraph of $H=K_{2k}\vee(K_{n-2k-1}\cup K_1)$. By Lemma \ref{lem2.1}, one has $\rho(H)=\rho_H>n-2$. Thus $\lambda_2<\rho_H$, as desired.

By plugging the value $\rho_H$ into $x$ of $\Phi(B_G,x)-\Phi(B_H,x)$, we have
$$
\Phi(B_G,\rho_H)-\Phi(B_H,\rho_H)=(s-2k)[\rho_H^2-s\rho_H+(s+1)n+2sk-2s^2-2k-4s-2].
$$
Let $\Psi_2(x)=x^2-sx+(s+1)n+2sk-2s^2-2k-4s-2$ be a real function in $x$, where $x\in[n-2,+\infty)$.
By some calculations, we find that $\Psi_2'(x)=2x-s$ and $\Psi_2''(x)=2>0$. Hence, $\Psi_2'(x)$ is a monotonically increasing function for $x\geqslant n-2$. Consequently, $\Psi_2'(x)\geqslant \Psi_2'(n-2)=2n-s-4>0$ for $x\geqslant n-2$. That is to say that $\Psi_2(x)$ is increasing on $[n-2,+\infty)$.
Recall that $\rho_H>n-2$, we have $\Psi_2(\rho_H)>\Psi_2(n-2)$.
By a simple calculation, one obtains
\begin{align}
\Psi_2(n-2)=&n^2-3n+2sk-2s^2-2k-2s+2\notag\\
        \geqslant&(2s-2k+2)^2-3(2s-2k+2)+2sk-2s^2-2k-2s+2\label{eq:4.02}\\
        =&2s^2-6ks+4k^2-4k\notag\\
        \geqslant&2(2k+2)^2-6k(2k+2)+4k^2-4k\label{eq:4.03}\\
        =&8>0.\notag
\end{align}
The inequality in \eqref{eq:4.02} follows by $\frac{3}{2}<2k+6\leqslant2s-2k+2\leqslant n$, and the inequality in \eqref{eq:4.03} follows by $\frac{3k}{2}<2k+2\leqslant s$. Thus, $\Psi_2(\rho_H)>\Psi_2(n-2)>0$. Together with $s\geqslant2k+2$, we have $\Phi(B_G,\rho_H)-\Phi(B_H,\rho_H)=(s-2k)\Psi_2(\rho_H)>0$. Note that $\Phi(B_H,\rho_H)=0$, one has $\Phi(B_G,\rho_H)>0$, as desired.

Therefore, $\rho(G)<\rho(H)$ for $s\geqslant2k+2$. Then we proved $\rho(G)\leqslant\rho(H)$ with equality if and only if $G=H$. This contradicts the assumption.

This completes the proof of Theorem \ref{thm1.4}.
\end{proof}

\section{\normalsize Proof of Theorem \ref{thm1.6}}\setcounter{equation}{0}
In this section, we give the proof of Theorem \ref{thm1.6}, which presents a sufficient condition via the spectral radius of a graph to ensure that the graph is $1$-excludable.

\begin{proof}[\bf Proof of Theorem \ref{thm1.6}] Assume that $e$ is an edge of $G$. Suppose to the contrary that $G$ has no perfect matching excluding the edge $e$. By Lemma \ref{lem2.7}, there exists a vertex subset $S\subseteq V(G)$ such that the following statements hold:
\begin{wst}
\item[{\rm (1)}] If $G-S$ has a component containing an odd-bridge, then $o(G-S)>|S|-2$;
\item[{\rm (2)}] $o(G-S)>|S|$, otherwise.
\end{wst}
Choose such a connected graph $G$ so that its spectral radius is as large as possible. We proceed by considering the following two possible cases.

{\bf Case 1.} $G-S$ has no component containing an odd-bridge.

Note that $n$ is even and $o(G-S)\equiv|S| \pmod 2$. Thus, $o(G-S)\geqslant|S|+2$. By Lemmas \ref{lem2.1} and~\ref{lem2.10}, we may easily get contradictions for $n=6$ or $n\geqslant10$. For $n=8$, by a simple computation, we have $\rho(G)\leqslant\rho(K_1\vee(K_5\cup2K_1))\thickapprox5.0695<\rho(K_3\vee(K_2\cup3K_1))\thickapprox5.1757$, a contradiction.

{\bf Case 2.} $G-S$ has a component, say $\hat{G_1}$, containing an odd-bridge $e=uv$.

Note that $|V(\hat{G_1})|$ and $n$ are even. Then, $o(G-S)\equiv|S|\pmod 2$. Thus, $o(G-S)\geqslant|S|$. Firstly, we assume $S=\emptyset$, then $o(G)=o(G-S)\geqslant0$. Since $G$ is a connected graph and $n$ is even, one has $o(G)=0$. Then $G=\hat{G_1}$. Let $n_u$ and $n_v$ denote the number of vertices of connected components containing vertices $u$ and $v$ in $G-e$, respectively. Since $\delta(G)\geqslant2$, one has $n_u,\ n_v\geqslant3$. According to the choice of $G$ and Lemma \ref{lem2.1}, we may obtain $G=K_{n_u}+K_{n_v}$.

Let $F_1=K_2\vee(K_2\cup2K_1)$, $F_2=K_3\vee(K_2\cup3K_1)$ and $F_3=K_1\vee(K_{2}\cup K_{n-3})$. In what follows, we need to show the following claim to deduce a contradiction.
\begin{claim}\label{claim1}Assume $G$ is the graph described above.
\begin{wst}
\item[{\rm (a)}] If $n=6$, then $\rho(G)<\rho(F_1)$;
\item[{\rm (b)}] If $n=8$, then $\rho(G)<\rho(F_2)$;
\item[{\rm (c)}] If $n\geqslant10$, then $\rho(G)<\rho(F_3)$.
\end{wst}
\end{claim}
\begin{proof}[\bf Proof of Claim \ref{claim1}]\ Assume $G=K_{n_u}+K_{n_v}$, and $n_u$ and $n_v$ are two odd integers.

(a) If $n=6$, by a direct calculation, we have $\rho(F_1)\thickapprox3.6262>\rho(K_3+K_3)\thickapprox2.4142$. Together with Lemma \ref{lem2.11}, one has $\rho(G)\leqslant\rho(K_3+K_3)<\rho(F_1)$.

(b) If $n=8$, by a direct calculation, we have $\rho(F_2)\thickapprox5.1757>\rho(K_3+K_5)\thickapprox4.0615$. Together with Lemma \ref{lem2.11}, one has $\rho(G)\leqslant\rho(K_3+K_5)<\rho(F_2)$.

(c) We consider $n\geqslant10$. In view of \eqref{eq:2.1}, it is easy to check that
$$
\Phi_2(-2)=1+n>0,\ \ \Phi_2(-1)=-2n+8<0,\ \ \Phi_2(1)=4n-20>0,\ \  \Phi_2(3)=-14n+96<0,
$$
and
$$
\Phi_2(n-3)=(n-3)(n-3-\sqrt{5})(n-3+\sqrt{5})>0.
$$
Thus, one has $3<\rho(K_3+K_{n-3})<n-3$. Since $K_{n-2}$ is a proper subgraph of $F_3$, by Lemma \ref{lem2.1}, we have $n-3<\rho(F_3)$. Together with Lemma \ref{lem2.11}, we have $\rho(G)\leqslant\rho(K_3+K_{n-3})<n-3<\rho(F_3)$.

This completes the proof of Claim \ref{claim1}.
\end{proof}
In view of Claim \ref{claim1}, we may obtain a contradiction to {the assumption. }

Next, we assume $S\neq\emptyset$. Then $o(G-S)\geqslant|S|\geqslant1$.
According to the choice of $G$, one has
\begin{wst}
\item $G-S$ contains only one even connected component, say $\hat{G_1}$;
\item $G[S]$ (resp. each odd connected component of $G-S$) is a complete graph;
\item $G$ is the join of $G[S]$ and $G-S$, i.e., $G=G[S]\vee (G-S)$.
\end{wst}

For convenience, put $o(G-S)=q$, $|V(\hat{G_1})|=\hat{n}$ and $|S|=s$. Then $n\geqslant\hat{n}+s+q\geqslant2s+2$.
Assume that $G_1,G_2,\ldots,G_q$ are all the odd components of $G-S$. For convenience, let $n_i=|V(G_i)|$ for $i=1,2,\ldots,q$ and assume $n_1\geqslant n_2\geqslant\cdots\geqslant n_q$. Clearly, $n=s+\hat{n}+\sum_{i=1}^{q}n_i$ and $G=K_s\vee(\hat{G_1}\cup K_{n_1}\cup K_{n_2}\cup\ldots\cup K_{n_q})$.

Note that $\hat{G_1}$ is the unique even component of $G-S$ containing an odd-bridge $e=uv$. Assume that $G_u$ and $G_v$ are the two odd components of $\hat{G_1}-e$ containing vertices $u$ and $v$, respectively. According to the choice of $G$, we see that both $G_u$ and $G_v$ are complete graphs. Let $|V(G_u)|=n_u$ and $|V(G_v)|=n_v$. Then $\hat{G_1}=K_{n_v}+K_{n_u}$ and $\hat{n}=n_u+n_v$.
Let ${\bf x}$ be the Perron vector of $A(G)$ with respect to $\rho(G)$. Without loss of generality, assume that $x_u\geqslant x_v$. Suppose $N_G(v)\setminus N_G[u]\neq\emptyset$. Let
$$
H_1=G-\sum_{w\in N_G(v)\setminus N_G[u]}vw+\sum_{w\in N_G(v)\setminus N_G[u]}uw.
$$
By Lemma \ref{lem2.3}, we have that $\rho(G)<\rho(H_1)$. Note that $H_1$ is a subgraph of $H_2=K_s\vee(K_{\hat{n}-1}^+\cup K_{n_1}\cup K_{n_2}\cup\ldots\cup K_{n_q})$ and $H_2$ satisfies statement (1), a contradiction. Therefore, $N_G(v)\setminus N_G[u]=\emptyset$ and so $\hat{G_1}=K_{\hat{n}-1}^+$. Then $G=K_s\vee(K_{\hat{n}-1}^+\cup K_{n_1}\cup K_{n_2}\cup\ldots\cup K_{n_q})$.

For $s=1$, one has $n_q\geqslant3$. By Lemma \ref{lem2.8}, we have $\rho(G)\leqslant \rho(K_1\vee(K_{\hat{n}-1}^+\cup K_{n-\hat{n}-1-3(q-1)}\cup (q-1)K_3)$ with equality if and only if $G=K_1\vee(K_{\hat{n}-1}^+\cup K_{n-\hat{n}-1-3(q-1)}\cup (q-1)K_3)$.
Recall that $q\geqslant s$. It is clear that $K_1\vee(K_{\hat{n}-1}^+\cup K_{n-\hat{n}-1-3(q-1)}\cup (q-1)K_3)$ is a subgraph of $K_1\vee(K_{\hat{n}-1}^+\cup K_{n-\hat{n}-1})$ and $K_1\vee(K_{\hat{n}-1}^+\cup K_{n-\hat{n}-1})$ satisfies (1). According to the choice of $G$, one has $G=K_1\vee(K_{\hat{n}-1}^+\cup K_{n-\hat{n}-1})$. Next, we proceed by showing the following claim.

\begin{claim}\label{claim2}
Let $G=K_1\vee (K_{\hat{n}-1}^+\cup K_{n-\hat{n}-1})$ and assume $F_1, F_2, F_3$ are the graphs described as before.
\begin{wst}
\item[{\rm (a)}] If $n=6$, then $\rho(G)<\rho(F_1)$;
\item[{\rm (b)}] If $n=8$, then $\rho(G)<\rho(F_2)$;
\item[{\rm (c)}] If $n\geqslant10$, then $\rho(G)\leqslant\rho(F_3)$ with equality if and only if $G=F_3$.
\end{wst}
\end{claim}
\begin{proof}[\bf Proof of Claim \ref{claim2}] Note that $n-\hat{n}-1\geqslant3$. Then $2\leqslant\hat{n}\leqslant n-4$.

(a) If $n=6$, then $\hat{n}=2$ and so $G=K_1\vee(K_2\cup K_3).$  By a direct calculation, we have $\rho(K_1\vee(K_2\cup K_3))\thickapprox3.2618<3.6262\thickapprox\rho(F_1)$, as desired.

(b) If $n=8$, then $\hat{n}=2,4$ and so $G\in\{K_1\vee(K_2\cup K_5), K_1\vee(K_3^+\cup K_3)\}$. By some computations, we have $\rho(K_1\vee(K_2\cup K_5))\thickapprox5.0874$, $\rho(K_1\vee(K_3^+\cup K_3))\thickapprox3.8704$ and $\rho(F_2)\thickapprox5.1757$, as desired.

(c) Clearly, $G=F_3$ if $\hat{n}=2$. In what follows, we are show $\rho(G)<\rho(F_3)$ for $\hat{n}\geqslant4$.

According to the partition $V(G)=V(K_1)\cup V(K_{\hat{n}-2})\cup\{u\}\cup\{v\}\cup V(K_{n-\hat{n}-1})$, the quotient matrix of $A(G)$ is
$$
Q_G=\left(
  \begin{array}{ccccc}
    0 & \hat{n}-2 & 1 & 1 & n-\hat{n}-1 \\
    1 & \hat{n}-3 & 1 & 0 & 0\\
    1 & \hat{n}-2 & 0 & 1 & 0\\
    1 & 0 & 1 & 0 &0 \\
    1 & 0 & 0 & 0 &n-\hat{n}-2
  \end{array}
\right).
$$
It is easy to check that the characteristic polynomial of $Q_G$ is
\begin{align*}
\Phi(Q_G,x)=&x^5-(n-5)x^4+(n\hat{n}-\hat{n}^2-4n+8)x^3+(3n\hat{n}-3\hat{n}^2-4n-2)x^2\\
             &+(6n-3\hat{n}-19)x-3\hat{n}n+3\hat{n}^2+9n-4\hat{n}-15.
\end{align*}
Clearly, the partition $V(G)=V(K_1)\cup V(K_{\hat{n}-2})\cup\{u\}\cup\{v\}\cup V(K_{n-\hat{n}-1})$ is equitable. By Lemma \ref{lem2.4}, the largest root, say $\eta_1$, of $\Phi(Q_G,x)=0$ equals $\rho(G)$. By plugging the value $2$ into $\hat{n}$ of $\Phi(Q_G,x)$, we have
$$
\Phi(Q_{F_3},x)=x^5-(n-5)x^4-(2n-4)x^3-(14-2n)x^2-(25-6n)x+3n-11.
$$
Obviously, the largest root of $\Phi(Q_{F_3},x)=0$ equals $\rho(F_3)$.
By plugging the value $\eta_1$ into $x$ of $\Phi(Q_{F_3},x)-\Phi(Q_G,x)$, we have
$$
\Phi(Q_{F_3},\eta_1)-\Phi(Q_G,\eta_1)=(\hat{n}-2)[-(n-\hat{n}-2)\eta_1^3-(3n-3\hat{n}-6)\eta_1^2+3\eta_1+3n-3\hat{n}-2].
$$
Let $f_1(x)=-(n-\hat{n}-2)x^3-(3n-3\hat{n}-6)x^2+3x+3n-3\hat{n}-2$ be a real function in $x$, where $x\in[n-\hat{n}-1,+\infty)$.

Consider the second derivative of $f_1(x)$ for $x\in [n-\hat{n}-1,+\infty)$. We have
$f_1''(x)=-6(x+1)(n-\hat{n}-2)<0$ for $x\geqslant n-\hat{n}-1$.
Therefore, $f_1'(x)$ is a decreasing function for $x\in[n-\hat{n}-1,+\infty)$. Consequently, for $x\geqslant n-\hat{n}-1$, we have
$f_1'(x)\leqslant f_1'(n-\hat{n}-1)=-3n^3+(9\hat{n}+6)n^2+(-9\hat{n}^2-12\hat{n}+3)n+3\hat{n}^3+ 6\hat{n}^2-3\hat{n}-3.$ Let $f_2(x)=-3x^3+(9\hat{n}+6)x^2+(-9\hat{n}^2-12\hat{n}+3)x+3\hat{n}^3+6\hat{n}^2-3\hat{n}-3$ be a real function in $x$, where $x\in[\hat{n}+4,+\infty).$ Consider the second derivative of $f_2(x)$ for $x\in [\hat{n}+4,+\infty)$, we have $f''_2(x)=-18x+18\hat{n}+12<0$. Then, $f_2'(x)$ is a decreasing function for $x\in[\hat{n}+4,+\infty)$. Thus, for $x\geqslant \hat{n}+4$, we have $f_2'(x)\leqslant f_2'(\hat{n}+4)=-93<0$. Therefore, $f_2(x)$ is a decreasing function for $x\in[\hat{n}+4,+\infty)$. Then $f_1'(x)\leqslant f_2(n)\leqslant f_2(\hat{n}+4)=-87<0$. Then $f_1(x)$ is a decreasing function for $x\in[n-\hat{n}-1,+\infty)$.  Therefore, $f_1(x)\leqslant f_1(n-\hat{n}-1)=-n^4+(4\hat{n}+2)n^3+(-6\hat{n}^2-6\hat{n}+3)n^2+(4\hat{n}^3+6\hat{n}^2-6\hat{n}-2)n-\hat{n}^4-2\hat{n}^3+3\hat{n}^2+2\hat{n}-1$.
Let $f_3(x)=-x^4+(4\hat{n}+2)x^3+(-6\hat{n}^2-6\hat{n}+3)x^2+(4\hat{n}^3+6\hat{n}^2-6\hat{n}-2)x-\hat{n}^4-2\hat{n}^3+3\hat{n}^2+2\hat{n}-1$ be a real function in $x$. Through some computations, we have
$$
f_3(-\infty)<0,\ f_3(\hat{n}-1)=1,\ f_3(\hat{n})=-1,\ f_3(\hat{n}+2)=7,\ f_3(\hat{n}+4)=-89.
$$
Thus, $f_3(x)$ is a decreasing function for $x\in[\hat{n}+4,+\infty)$. Thus $f_3(n)\leqslant f_3(\hat{n}+4)<0$.
Note that $K_{n-\hat{n}}$ is a subgraph of $G$. By Lemma \ref{lem2.1}, we have $\rho(G)=\eta_1>n-\hat{n}-1=\rho(K_{n-\hat{n}})$. Notice that $f_1(n-\hat{n}-1)=f_3(n)<0$. Then $\Phi(Q_{F_3},\eta_1)-\Phi(Q_G,\eta_1)=(\hat{n}-2)f_1(\eta_1)<(\hat{n}-2)f_1(n-\hat{n}-1)<0$. Since $\Phi(Q_G,\eta_1)=0$, we have $\Phi(Q_{F_3},\eta_1)<0$. Thus $\rho(G)<\rho(F_3)$, as desired.
\end{proof}

By Claim \ref{claim2}, we get contradictions. Next, we consider $s\geqslant2$.

If $s\geqslant2$, then $n_q\geqslant1$. By Lemma \ref{lem2.8}, we have $\rho(G)\leqslant \rho(K_s\vee(K_{\hat{n}-1}^+\cup K_{n-s-\hat{n}-q+1}\cup(q-1)K_1))$ with equality if and only if $G=K_s\vee(K_{\hat{n}-1}^+\cup K_{n-s-\hat{n}-q+1}\cup(q-1)K_1)$. Notice that $o(K_s\vee(K_{\hat{n}-1}^+\cup K_{n-s-\hat{n}-q+1}\cup(q-1)K_1)-V(K_s))=o(G-S)\geqslant|S|$. According to the choice of $G$, we have $G=K_s\vee(K_{\hat{n}-1}^+\cup K_{n-s-\hat{n}-q+1}\cup(q-1)K_1)$.

Recall that $q\geqslant s$. It is clear that $K_s\vee(K_{\hat{n}-1}^+\cup K_{n-s-\hat{n}-q+1}\cup(q-1)K_1)$ is a subgraph of $K_s\vee(K_{\hat{n}-1}^+\cup K_{n-2s-\hat{n}+1}\cup(s-1)K_1)$. Together with Lemma \ref{lem2.1} and the choice of $G$, one has $G=K_s\vee(K_{\hat{n}-1}^+\cup K_{n-2s-\hat{n}+1}\cup(s-1)K_1)$.

In what follows, we show that $n-2s-\hat{n}=0$ or $\hat{n}=2$. Suppose that $n-2s-\hat{n}>0$ and $\hat{n}>2$. The vertex set of $G$ may be partitioned as $V(G)=V(K_s)\cup V(K_{\hat{n}-1}-u)\cup V(K_{n-2s-\hat{n}+1})\cup V((s-1)K_1)\cup \{u\}\cup \{v\}$, where $V(K_s)=\{r_1,\ldots,r_s\}$, $V((s-1)K_1)=\{t_1,\ldots,t_{s-1}\}$, $V(K_{\hat{n}-1}-u)=\{u_1,\ldots,u_{\hat{n}-2}\}$, $V(K_{n-2s-\hat{n}+1})=\{v_1,\ldots,v_{n-2s-\hat{n}+1}\}$. Let ${\bf x}$ be the Perron vector of $A(G)$ with respect to $\rho(G)=:\rho$. By Lemma \ref{lem2.2}, ${\bf x}$ takes the same value (say $x_0,x_1,x_2,x_3,x_u$ and $x_v$) on the vertices of $V(K_s),V(K_{n-2s-\hat{n}+1}),V((s-1)K_1),V(K_{\hat{n}-1}-u),u$ and $v$, respectively.
For convenience, put $n-2s-\hat{n}=t$. Let
$$
H=G+\sum_{i=2}^{t+1}v_iu+\sum_{i=1}^{\hat{n}-2} \ \sum_{j=2}^{t+1}u_iv_j-\sum_{i=2}^{t+1}v_iv_1.
$$
Clearly, $H=K_s\vee(K_{n-2s-1}^+\cup sK_1)$ and $H$ satisfies statement (1).
Let ${\bf y}$ be the Perron vector of $A(H)$ corresponding to $\rho(H)=:\rho'$. By Lemma \ref{lem2.2}, ${\bf y}$ takes the same values $y_0,y_1,y_2,y_u$ and $y_v$ on the vertices of $V(K_s), V(sK_1),V(K_{n-2s-1}-u),u$ and $v$, respectively.
By $A(H){\bf y}=\rho'{\bf y}$, we have
\begin{align*}
\rho'y_u=&sy_0+y_v+(n-2s-2)y_2,\ \ \ \ \rho'y_1=sy_0.
\end{align*}
Recall that $n-2s\geqslant\hat{n}+2$. Then $y_u>y_1$.
Together with $A(G){\bf x}=\rho {\bf x}$ and $A(H){\bf y}=\rho'{\bf y}$, we have
\begin{align*}
{\bf y}^T(\rho'-\rho){\bf x}=&{\bf y}^T(A(H)-A(G)){\bf x}\\
                            =&\sum_{i=2}^{t+1}(x_{v_i}y_u+x_uy_{v_i})+\sum_{i=1}^{\hat{n}-2}\  \sum_{j=2}^{t+1}(x_{u_i}y_{v_j}+x_{v_j}y_{u_i})
                -\sum_{i=2}^{t+1}(x_{v_i}y_{v_1}+x_{v_1}y_{v_i})\\
                =&t(x_1y_u+x_uy_2)+t(\hat{n}-2)(x_3y_2+x_1y_2)-t(x_1y_1+x_1y_2)\\
                =&t[x_1y_u+x_uy_2+(\hat{n}-2)(x_3y_2+x_1y_2)-(x_1y_1+x_1y_2)]\\
                >&t[x_1(y_u-y_1)+(\hat{n}-2)x_3y_2+(\hat{n}-3)x_1y_2]\tag{As $x_u>0$ and $y_2>0$}\\
                >&t[(\hat{n}-2)x_3y_2+(\hat{n}-3)x_1y_2]>0.\tag{As $y_u>y_1$}
\end{align*}
Consequently, $\rho'-\rho>0$ when $\hat{n}\geqslant4$ and $n-2s-\hat{n}\geqslant2$, a contradiction to the choice of $G$.

Therefore, $n-2s-\hat{n}=0$ or $\hat{n}=2$. Thus $G=K_s\vee(K_{n-2s-1}^+\cup sK_1)$ or $K_s\vee(K_2\cup K_{n-2s-1}\cup (s-1)K_1)$.
Next, we proceed by showing the following claim.
\begin{claim}\label{claim3}Let $G=K_s\vee(K_2\cup K_{n-2s-1}\cup (s-1)K_1)$ and let $F_i\ ( i\in\{1,2,3\})$ be the graphs defined as before.
\begin{wst}
\item[{\rm (a)}] If $n=6$, then $G=F_1$;
\item[{\rm (b)}] If $n=8$, then $\rho(G)\leqslant\rho(F_2)$ with equality if and only if $G=F_2$;
\item[{\rm (c)}] If $n\geqslant10$, then $\rho(G)<(F_3)$.
\end{wst}
\end{claim}
\begin{proof}[\bf Proof of Claim \ref{claim3}] Recall that $n\geqslant2s+2$ and $s\geqslant2$. Then $2\leqslant s\leqslant\frac{n-2}{2}$.

(a) If $n=6$, then $s=2$ and so $G=F_1$, as desired.

(b) If $n=8$, then $s=2,3$ and so $G\in\{K_2\vee(K_2\cup K_3\cup K_1),F_2\}$. By some computations, we have $\rho(K_2\vee(K_2\cup K_3\cup K_1))\thickapprox4.7131$ and $\rho(F_2)\thickapprox5.1757$, as desired.

(c) For $n\geqslant10$, we proceed by considering the following two cases.

{\bf Case 1.} $n=2s+2$. In this case, one has $G=K_s\vee(K_2\cup sK_1)$ and $s\geqslant4$. Next, we shall show that $\rho(G)<2s-1<\rho(F_3)$ for $s\geqslant4$ ($n\geqslant10$).
According to the partition $V(K_s\vee(K_2\cup sK_1))=V(K_s)\cup V(K_2)\cup V(sK_1)$, the quotient matrix of $A(G)$ is
$$
N_G=
\left(
  \begin{array}{ccc}
  s-1 &2 &s\\
  s &1 &0 \\
  s &0 &0 \\
  \end{array}
\right).
$$
By a simple calculation, the characteristic polynomial of $N_G$ is
$\Phi(N_G,x)=x^3-sx^2-(s^2+s+1)x+s^2.$
Clearly, the vertex partition of $V(G)$ is equitable. By Lemma \ref{lem2.4}, the largest root, say $\eta_2$, of $\Phi(N_G,x)=0$ equals $\rho(G)$.
By a simple computation, we have
\[\label{eq:5.1}
\Phi(N_G,-\infty)<0,\ \Phi(N_G,0)=s^2>0,\ \Phi(N_G,s)=-s^3-s<0,
\]
and
$$\Phi(N_G,2s-1)=2s^3-8s^2+4s=2s(s^2-4s+2)=2s(s-2-\sqrt{2})(s-2+\sqrt{2})>0.$$
Consequently, we have $s<\rho(G)=\eta_2<2s-1$. Note that $K_{2s}$ is a proper subgraph of $F_3$, by Lemma~\ref{lem2.1}, we have $\rho(F_3)>2s-1=\rho(K_{2s})$. Thus, $\rho(G)<2s-1<\rho(F_3)$, as desired.

{\bf Case 2.} $n\geqslant2s+4$. According to the partition $V(G)=V(K_s)\cup V(K_{2})\cup V(K_{n-2s-1})\cup V((s-1)K_1)$, the quotient matrix of $A(G)$ is
$$
T_G=
\left(
  \begin{array}{cccc}
  s-1 &2 &n-2s-1 &s-1\\
  s &1 &0 &0\\
  s &0 &n-2s-2 &0\\
  s &0 &0 &0\\
  \end{array}
\right).
$$
By a direct computation, the characteristic polynomial of $T_G$ is
\begin{align*}
\Phi(T_G,x)=&x^4+(s+2-n)x^3-(s^2+s+1)x^2+(s^2n-2s^3+sn-3s^2+n-4s-2)x\\
&-s^2n+2s^3+sn-2s.
\end{align*}
Obviously, the partition $V(G)=V(K_s)\cup V(K_{2})\cup V(K_{n-2s-1})\cup V((s-1)K_1)$ is equitable. By Lemma~\ref{lem2.4}, the largest root, say $\eta_3$, of $\Phi(T_G,x)=0$ satisfies $\eta_3=\rho(G)$.

According to the partition $V(F_3)=V(K_1)\cup V(K_2)\cup V(K_{n-3})$, the quotient matrix of $A(F_3)$ is
$$
\left(
  \begin{array}{ccc}
    0 & 2 & n-3 \\
    1 & 1 & 0 \\
    1 & 0 & n-4 \\
  \end{array}
\right).
$$
It is easy to check that the characteristic polynomial is
\[\label{eq:5.2}
\Phi(F_3,x)=x^3+(3-n)x^2-3x+3n-11.
\]
Clearly, the partition $V(F_3)=V(K_1)\cup V(K_2)\cup V(K_{n-3})$ is equitable. By Lemma \ref{lem2.4}, the largest root of $\Phi(F_3,x)=0$ equals $\rho(F_3)$.
Next, we are to show $\Phi(T_{F_3},\eta_3)<0$ for $s\geqslant2$ in order to prove $\rho(G)<\rho(F_3)$.

By plugging the value $\eta_3$ into $x$ of $\psi_1(x)=x\Phi(T_{F_3},x)-\Phi(T_G,x)$, we have
$$
\psi_1(\eta_3)=(s-1)[-\eta_3^3+(s+2)\eta_3^2+(2s^2-sn-2n+5s+9)\eta_3+sn-2s^2-2s].
$$
Let $g_1(x)=-x^3+(s+2)x^2+(2s^2-sn-2n+5s+9)x+sn-2s^2-2s$ be a real function in $x$, where $x\in[n-s-2,+\infty)$.

Consider the second derivative of $g_1(x)$ for $x\in [n-s-2,+\infty)$. We have
\begin{align*}
g_1''(x)=&-6x+2s+4\\
        \leqslant&-6(n-s-2)+2s+4\tag{As $x\geqslant n-s-2$}\\
        =&-6n+8s+16\\
        \leqslant&-6(2s+4)+8s+16\tag{As $n\geqslant2s+4$}\\
        =&-4s-8<0.
\end{align*}
Therefore, $g_1'(x)$ is a decreasing function for $x\in[n-s-2,+\infty)$. Consequently, for $x\geqslant n-s-2$, we have
$$g_1'(x)\leqslant g_1'(n-s-2)=-3n^2+(7s+14)n-3s^2-15s-11.$$
Let $g_2(x)=-3x^2+(7s+14)x-3s^2-15s-11$ be a real function in $x$, where $x\in[2s+4,+\infty)$. By a simple computation, we have
$$
g_2(0)=-3s^2-15s-11<0,\  g_2(s+2)=s^2+s+5>0
$$
and
$$
g_2(2s+4)=-s^2-7s-3<0.
$$
Thus $g_2(x)$ is a monotonically decreasing function for $x\in[2s+4,+\infty)$. Then $g_1'(x)\leqslant g_1'(n-s-2)=g_2(n)\leqslant g_2(2s+4)<0$ for $x\in[n-s-2,+\infty)$. Consequently, $g_1(x)$ is a monotonically decreasing function for $x\in[n-s-2,+\infty)$.

Now we prove that $g_1(\eta_3)<0$.
Recall that $\rho(G)=\eta_3>n-s-2$. Together with the fact that $g_1(x)$ is decreasing on $x\in[n-s-2,+\infty)$, we have $g_1(\eta_3)<g_1(n-s-2)=-n^3+(3s+6)n^2-(2s^2+10s+7)n+s^2+3s-2.$

Let $g_3(x)=-x^3+(3s+6)x^2-(2s^2+10s+7)x+s^2+3s-2$ be a real function in $x$, where $x\in[2s+4,+\infty).$  It is routine to check that the derivative of $g_3(x)$ is $$g_3'(x)=-3x^2+2(3s+6)x-2s^2-10s-7.$$ By a simple computation, we have
$$g_3'(-\infty)<0,\  g_3'(s)=s^2+2s-7>0,\  g_3'(2s+4)=-2s^2-10s-7<0$$

Thus, $g_3(x)$ is a monotonically decreasing function for $x\in[2s+4,+\infty)$. Thus $g_1(\eta_3)<g_3(n)\leqslant g_3(2s+4)=-s^3-3s+1<0$, as desired.
Then $\psi_1(\eta_3)=\eta_3\Phi(T_{F_3},\eta_3)-\Phi(T_G,\eta_3)=(s-1)g_1(\eta_3)<0.$ Since $\Phi(T_G,\eta_3)=0$. Thus $\Phi(T_{F_3},\eta_3)<0$ and so $\rho(G)<\rho(F_3)$.

Together with Cases 1 and 2, we complete the proof of (c).
\end{proof}

For convenience, let $W_1=K_{\frac{n-2}{2}}\vee(K_2\cup \frac{n-2}{2}K_1)$ and $W_2=K_2\vee(K_{n-5}^+\cup 2K_1)$.
Then we have the following claim.
\begin{claim}\label{claim4}Let $G=K_s\vee(K_{n-2s-1}^+\cup sK_1)$ and let $W_1$, $W_2$ be the graphs defined as above.
\begin{wst}
\item[{\rm (a)}] If $6\leqslant n\leqslant 12$, then $\rho(G)\leqslant\rho(W_1)$ with equality if and only if $G=W_1$;
\item[{\rm (b)}] If $n\geqslant14$, then $\rho(G)\leqslant\rho(W_2)$ with equality if and only if $G=W_2$.
\end{wst}
\end{claim}
\begin{proof}[\bf Proof of Claim \ref{claim4}] Note that $s\geqslant2$ and $n-2s-1\geqslant1$, we have $2\leqslant s\leqslant\frac{n-2}{2}.$ Similar to the proofs of (a) and (b) of Claim 3, by direct calculations, one has (a) is true. Next, we prove (b) by considering the following three cases.

{\bf Case 1.} $n\geqslant2s+6$.\ Clearly, $G=W_2$ when $s=2$. In what follows, we are to show $\rho(G)<\rho(W_2)$ for $s\geqslant3$. According to the partition $V(G)=V(K_s)\cup V(K_{n-2s-2})\cup \{u\}\cup \{v\}\cup V(sK_1)$, the quotient matrix of $A(G)$ is
$$
P_G=
\left(
  \begin{array}{ccccc}
  s-1 &n-2s-2 &1 &1 &s\\
  s &n-2s-3 &1 &0 &0\\
  s &n-2s-2 &0 &1 &0\\
  s &0 &1 &0 &0\\
  s &0 &0 &0 &0\\
  \end{array}
\right).
$$
By a direct computation, the characteristic polynomial of $P_G$ is
\begin{align}\label{eq:5.3}
\Phi(P_G,x)=&x^5-(n-s-4)x^4-(s^2+2n-s-4)x^3+(s+1)(ns-2s^2-3s-2)x^2\notag\\
&-(2s^3-ns^2+3s^2-ns+5s-n+3)x-ns^2+2s^3+3s^2.
\end{align}
Obviously, the partition $V(G)=V(K_s)\cup V(K_{n-2s-2})\cup \{u\}\cup \{v\}\cup V(sK_1)$ is equitable. By Lemma~\ref{lem2.4}, the largest root, say $\eta_4$, of $\Phi(P_G,x)=0$ satisfies $\eta_4=\rho(G)$.
Note that $K_{n-s-1}$ is a proper subgraph of $G=K_s\vee(K_{n-2s-1}^+\cup sK_1)$, by Lemma \ref{lem2.1}, we have $\eta_4>n-s-2$.

According to the partition $V(W_2)=V(K_2)\cup V(K_{n-6})\cup \{u\}\cup \{v\}\cup V(2K_1)$, the quotient matrix of $A(W_2)$ is
$$
P_{W_2}=
\left(
  \begin{array}{ccccc}
  1 &n-6 &1 &1 &2\\
  2&n-7 &1 &0 &0\\
  2 &n-6 &0 &1 &0\\
  2 &0 &1 &0 &0\\
  2 &0 &0 &0 &0\\
  \end{array}
\right).
$$
It is easy to check that the characteristic polynomial of $P_{W_2}$ is
\[\label{eq:5.4}
\Phi(P_{W_2},x)=x^5-(n-6)x^4-(2n-2)x^3-(48-6n)x^2-(41-7n)x-4n+28.
\]
Clearly, the partition $V(F_1)=V(K_2)\cup V(K_{n-6})\cup \{u\}\cup \{v\}\cup V(2K_1)$ is equitable. By Lemma \ref{lem2.4}, the largest root of $\Phi(P_{W_2},x)=0$ equals $\rho(W_2)$. 

By plugging the value $\eta_4$ into $x$ of $\psi_2(x)=\Phi(P_{W_2},x)-\Phi(P_G,x)$, we have
\begin{align*}
\psi_2(\eta_4)=&(s-2)[-\eta_4^4+(s+1)\eta_4^3+(-sn+2s^2-3n+9s+23)\eta_4^2+(-sn+2s^2-3n+7s+19)\eta_4\\
&+sn-2s^2+2n-7s-14].
\end{align*}
Let $f_1(x)=-x^4+(s+1)x^3+(-sn+2s^2-3n+9s+23)x^2+(-sn+2s^2-3n+7s+19)x+sn-2s^2+2n-7s-14$ be a real function in $x$, where $x\in[n-s-2,+\infty)$. Now, we show that $f_1(x)$ is a monotonically decreasing function for $x\in[n-s-2,+\infty)$.

Consider the second derivative of $f_1(x)$ for $x\in [n-s-2,+\infty)$. We have
\begin{align*}
f_1''(x)=&-12x^2+6(s+1)x-2ns+4s^2-6n+18s+46\\
        \leqslant&-12(n-s-2)^2+6(s+1)(n-s-2)-2ns+4s^2-6n+18s+46\\
        =&-12n^2+(28s+48)n-14s^2-48s-14\\
        \leqslant&-12(2s+6)^2+(28s+48)(2s+6)-14s^2-48s-14\\
        =&-6s^2-72s-158<0.
\end{align*}
The first inequality follows by $\frac{6(s+1)}{2\times12}=\frac{s+1}{4}<s+4\leqslant n-s-2$ and the second one follows by $\frac{28s+48}{2\times12}=\frac{7s+12}{6}<2s+6\leqslant n$.
Therefore, $f_1'(x)$ is a decreasing function for $x\in[n-s-2,+\infty)$. Consequently, for $x\geqslant n-s-2$, we have
$$f_1'(x)\leqslant f_1'(n-s-2)=-4n^3+(13s+21)n^2-(12s^2+39s+5)n+3s^3+15s^2-3s-29.$$
Let $f_2(x)=-4x^3+(13s+21)x^2-(12s^2+39s+5)x+3s^3+15s^2-3s-29$ be a real function in $x$, where $x\in[2s+6,+\infty)$. By a simple computation, we have
$$
f_2(-\infty)>0,\  f_2(s)=-3s^2-8s-29<0,\  f_2(s+2)=s^2+2s+13>0,
$$
and
$$
f_2(2s+6)=-s^3-27s^2-139s-167<0
$$
for $s\geqslant3$.
Thus $f_2(x)$ is a monotonically decreasing function for $x\in[2s+6,+\infty)$. Then $f_1'(x)\leqslant f_1'(n-s-2)=f_2(n)\leqslant f_2(2s+6)<0$ for $x\in[n-s-2,+\infty)$, as desired.

Now we prove that $f_1(\eta_4)<0$.
Recall that $\rho(G)=\eta_4>n-s-2$. Together with the fact that $f_1(x)$ is decreasing on $x\in[n-s-2,+\infty)$, we have $f_1(\eta_4)<f_1(n-s-2)=-n^4+(4s+6)n^3-(5s^2+15s-2)n^2+(2s^3+9s^2-13s-33)n+12s^2+36s+16$.

Let $f_3(x)=-x^4+(4s+6)x^3-(5s^2+15s-2)x^2+(2s^3+9s^2-13s-33)x+12s^2+36s+16$ be a real function in $x$, where $x\in[2s+6,+\infty).$  It is routine to check that the derivative of $f_3(x)$ is $$f_3'(x)=-4x^3+3(4s+6)x^2+2(-5s^2-15s+2)x+2s^3+9s^2-13s-33.$$ By a simple computation, we have
$$f_3'(-\infty)>0,\  f_3'(s+1)=-s^2-3s-15<0,\  f_3'(s+2)=s^2+3s+15>0$$
and
$$f_3'(2s+6)=-2s^3-39s^2-185s-225<0.$$
Thus, $f_3(x)$ is a monotonically decreasing function for $x\in[2s+6,+\infty)$. Thus $f_1(\eta_4)<f_3(n)\leqslant f_3(2s+6)=-6s^3-60s^2-168s-110<0$.

Note that $\Phi(P_G,\eta_4)=0.$ Hence, $\psi_2(\eta_4)=\Phi(P_{W_2},\eta_4)=(s-2)f_1(\eta_4)<0.$ Therefore, $\rho(G)<\rho(W_2).$

{\bf Case 2.} $n=2s+4$. In this case, $G=K_s\vee(K_3^+\cup sK_1)$ and $s\geqslant5$. By \eqref{eq:5.3}, we have
$$
\Phi(P_G,x)=x^5-sx^4-(s^2+3s+4)x^3+(s^2-s-2)x^2+(3s^2+s+1)x-s^2.
$$
By a simple computation, one has
$$\Phi(P_G,-\infty)<0,\ \Phi(P_G,-2)=5s^2+2s-10>0,\ \Phi(P_G,0)=-s^2<0,\ \Phi(P_G,1)=2s^2-4s-4>0,$$
$$\Phi(P_G,s+1)=-s^5-4s^4-8s^3-16s^2-14s-4<0,\ \Phi(P_G,2s)=s(8s^4-20s^3-30s^2-7s+2)>0$$
for $s\geqslant5$. Thus $s+1<\rho(G)<2s$. Note that $K_{n-3}$ is a proper subgraph of $W_2$. By Lemma \ref{lem2.1}, we have $\rho(W_2)>n-4=2s$. Thus $\rho(G)<2s<\rho(W_2)$ for $s\geqslant5$, as desired.

{\bf Case 3.} $n=2s+2$. In this case, one has $G=K_s\vee(K_2\cup sK_1)$ and $s\geqslant6$. It is easy to straightly verify that
$$
\Phi(N_G,2s-2)=2s^3-15s^2+20s-6>0
$$
for $s\geqslant6$. Together with \eqref{eq:5.1}, one has $s<\rho(G)<2s-2$. Note that $K_{n-3}$ is a proper subgraph of $W_2$. By Lemma \ref{lem2.1}, we have $\rho(W_2)>n-4=2s-2$. Thus, $\rho(G)<2s-2<\rho(W_2)$, as desired.

Together with Cases 1, 2 and 3, we complete the proof of (c).
\end{proof}

\begin{claim}\label{claim5}
If $n\geqslant10$, then $\rho(F_3)>\rho(W_2)$.
\end{claim}
\begin{proof}[\bf Proof of Claim \ref{claim5}] For convenience, put $\rho(W_2)=\rho$. Next we shall show $\Phi(T_{F_3},\rho)<0$ to prove $\rho(F_3)>\rho(W_2)$.
By \eqref{eq:5.2} and \eqref{eq:5.4}, we have
$$
\rho^2\Phi(T_{F_3},\rho)-\Phi(P_{W_2},\rho)=-3\rho^4+(2n-5)\rho^3+(37-3n)\rho^2+(41-7n)\rho+4n-28.
$$
Let $h_1(x)=-3x^4+(2n-5)x^3+(37-3n)x^2+(41-7n)x+4n-28$ be a real function in $x$, where $x\in[n-4,+\infty)$. Next, we show that $h_1(x)$ is a monotonically decreasing function for $x\in[n-4,+\infty)$.

Consider the second derivative of $h_1(x)$ for $x\in [n-4,+\infty)$. Then we have
\begin{align*}
h_1''(x)=&-36x^2+6(2n-5)x-6n+74\\
        \leqslant&-36(n-4)^2+6(2n-5)(n-4)-6n+74\tag{As $\frac{6(2n-5)}{2\times36}<n-4$}\\
        =&-24n^2+204n-382<0
\end{align*}
for $n\geqslant10$. Therefore, $h_1'(x)$ is a decreasing function for $x\in[n-4,+\infty)$. Consequently, for $x\geqslant n-4$, we have
$h_1'(x)\leqslant f_1'(n-4)=-6n^3+75n^2-269n+273<0$ for $n\geqslant10$. Then $h_1(x)$ is a decreasing function for $x\in[n-4,+\infty)$. Recall that $\rho=\rho(W_2)<n-4$. Then  $h_1(\rho)<h_1(n-4)=-n^4+16n^3-78n^2+129n-48<0$ for $n\geqslant10$. Therefore, $\rho^2\Phi(T_{F_3},\rho)-\Phi(P_{W_2},\rho)<0.$ Note that $\Phi(P_{W_2},\rho)=0.$ Then $\Phi(T_{F_3},\rho)<0$ and so $\rho(F_3)>\rho(W_2)$ for $n\geqslant10$.
\end{proof}

Clearly, $F_1=W_1$ when $n=6$ and $F_2=W_1$ when $n=8$. By Claims \ref{claim3} and \ref{claim4}, we get contradictions for $n=6,8$. Note that $K_{s}\vee(K_2\cup K_{n-2s-1}\cup(s-1)K_1)=W_1$ when $s=\frac{n-2}{2}$.  By Claim \ref{claim3}, one has $\rho(W_1)<\rho(F_3)$ for $n\geqslant10$. Together with Claims \ref{claim3} and \ref{claim4}, we have $\rho(G)\leqslant\rho(F_3)$ for $n=10,12$, contradictions. By Claims \ref{claim3}, \ref{claim4} and \ref{claim5}, we get contradictions for $n\geqslant14$.

By Cases 1 and 2, our result holds.
\end{proof}
\section{\normalsize Concluding remarks}
In this paper, we consider some structure condition or spectral condition to guarantee the existence of some substructure in the graph. We obtain a sharp lower bound, respectively, on the size and the spectral radius of a graph $G$ to ensure that the graph $G$ is $k$-extendable. All the corresponding extremal graphs are identified fully. Furthermore, we establish a sharp lower bound, respectively, on the size and the spectral radius of a graph $G$ to guarantee that the graph $G$ is $1$-excludable. All the corresponding extremal graphs are characterized completely.

There are still some other interesting problems to be considered along the above line. For example, if we focus on the bipartite graph $G$, how to establish a sharp lower bound on the size (resp. spectral radius) of $G$ to ensure it is $k$-extendable, or $1$-excludable. Similarly, if we focus on regular graph, how to establish a sharp lower bound on the size (resp. spectral radius) of a regular graph to ensure it is $k$-extendable, or $1$-excludable.

We will do them in the near future.

\section*{\normalsize Declarations}
\noindent{\bf Authors' contributions}\vspace{3mm}

Both authors contributed equally to all aspects of the work.\vspace{4mm}

\noindent{\bf Data availability statement}\vspace{3mm}

Data sharing not applicable to this article as no datasets were generated or analysed during the current study.\vspace{4mm}

\noindent{\bf Conflict of interests statement}\vspace{3mm}

The authors declare that there is no conflict of interests regarding this publication.


\begin{thebibliography}{99}
\small \setlength{\itemsep}{-.2mm}
\bibitem{A.K}J. Akiyama, M. Kano, Factors and factorizations of graphs. Proof techniques in factor theory, Lecture Notes in Mathematics, 2031. Springer, Heidelberg, 2011.
\bibitem{AFS2020}R.E.L. Aldred, J. Fujisawa, A. Saito, Distance matching extension and local structure of graphs, J. Graph Theory 93 (1) (2020) 5-20.
\bibitem{AFS2021}R.E.L. Aldred, J. Fujisawa, A. Saito, Distance matching extension in cubic bipartite graphs, Graphs Combin. 37 (5) (2021) 1793-1806.
\bibitem{APR2020}R.E.L. Aldred, M.D. Plummer, W. Ruksasakchai, Distance restricted matching extension missing vertices and edges in 5-connected triangulations of the plane. J. Graph Theory 95 (2) (2020) 240-255.
\bibitem{FS2020}J. Fujisawa, H. Seno, Edge proximity and matching extension in projective planar graphs, J. Graph Theory 95 (3) (2020) 341-367.
\bibitem{R.B.}R.B. Bapat, Graphs and matrices, second edition, Universitex, Springer, London, 2014.


\bibitem{A-E}A.E. Brouwer, W.H. Haemers, Eigenvalues and perfect matchings, Linear Algebra Appl. 395 (2005) 155-162.


\bibitem{C.P.C}C.P. Chen, A note on matching extension (Chinese), J. Beijing Agri. Engin. Univ. 10 (1990) 293-297.

\bibitem{C.P}C.P. Chen, Binding number and toughness for matching extension, Discrete Math. 146 (1995) 303-306.


\bibitem{Cpc}C.P. Chen, G.Z. Liu, Toughness of graphs and $[a,b]$-factors with prescribed properties, J. {Combin.} Math. {Combin}. Comput. 12 (1992) 215-221.


\bibitem{C.G01}C. Godsil, G. Royle, Algebraic graph theory, Graduate Texts in Mathematics, 207. Springer-Verlag, New York, 2001.

\bibitem{HLZ2022}Y.F. Hao, S.C. Li, Q. Zhao, On the $A_\alpha$-spectral radius of graphs without large
matchings, Bull. Malays. Math. Sci. Soc. (2022) 45:3131-3156.
\bibitem{G.H}G. Hetyei, Rectangular configurations which can be covered by $2\times1$ rectangles, {P\'{e}csi Tan. F\"{o}isk. K\"{o}zl. 8 (1964) 351-367}.

\bibitem{M.A.}M.A. Henning, A. Yeo, Tight lower bounds on the size of a maximum matching in a regular graph, Graphs {Combin.} 23 (2007) 647-657.

\bibitem{R.A}R.A. Horn, C.R. Johnson, Matrix analysis, Cambridge University Press, Cambridge, 1985.

\bibitem{K2005}M. Kano, Qinglin Yu, Pan-fractional property in regular graphs, Electron. J. Combin. 12 (2005) {Note 23, 6 pp.}

\bibitem{K1993}P. Katerinis, Regular factors in regular graphs, Discrete Math. 113 (1993) 269-297.


\bibitem{LM}S.C. Li, S.J. Miao, Characterizing $\mathcal{P}_{\geqslant 2}$-factor and $\mathcal{P}_{\geqslant 2}$-factor covered graphs with respect to the size or the spectral radius, Discrete Math. 344 (2021) no. 11, 112588.

\bibitem{S.C.}S.C. Li, S.J. Miao, Complete characterization of odd factors via the size, spectral radius or distance spectral radius of graphs, Bull. Korean Math. Soc. 59 (4) (2022)  1045-1067.

\bibitem{C.H.C} C.H.C. Little, D.D. Grant, D.A. Holton, On defect-$d$ matchings in graphs, Discrete Math. 13 (1975) 41-54.

\bibitem{G.Z.}G.Z. Liu, On $[a,b]$-covered graphs, J. Combin. Math. Combin. Comput. 5 (1989) 14-22.

\bibitem{H.L}H. Liu, M. Lu, F. Tian, On the spectral radius of graphs with cut edges, Linear Algebra Appl. 389 (2004) 139-145.

\bibitem{H.M.}H. Minc, Nonnegative matrices, Wiley-Interscience Series in Discrete Mathematics and Optimization. A Wiley-Interscience Publication. John Wiley \& Sons, Inc., New York, 1988.

\bibitem{0007} V. Nikiforov, Merging the $A$- and $Q$-spectral theories, Appl. Anal. Discrete Math. 11 (2017) 81-107.

\bibitem{S.O.}S. O, Spectral radius and matchings in graphs, Linear Algebra Appl. 614 (2021) 316-324.

\bibitem{P1974}J. Plesn\'{\i}k, Remarks on regular factors of regular graphs, Czechoslovak Math. J. 24 (99) (1974) 292-300.

\bibitem{M.D.}M.D. Plummer, On $n$-extendable graphs, Discrete Math. 31 (1980) 201-210.





\bibitem{MDP1994}M.D. Plummer, Extending matchings in graphs: A survey, Discrete Math. 127 (1994) 277-292.

\bibitem{D-M-2008}M.D. Plummer, Recent progress in matching extension, Building Bridges, Bolyai Society Mathematical
Studies (M. Gr\"{o}tschel and G. O. H. Katona, eds.), vol. 19, Springer-Verlag, Berlin, 2008, pp. 427-454.


\bibitem{D.B.}D.B. West, Introduction to Graph Theory, Prentice Hall, Inc., Upper Saddle River, NJ, 2001.
\bibitem{YLH}L.H. You, M. Yang, W. So, W.G. Xi, On the spectrum of an equitable quotient matrix and its application, Linear Algebra Appl. 577 (2019) 21-40.

\bibitem{Y.L}Q.L. Yu, G.Z. Liu, Graph factors and matching extensions, Higher Education Press, Beijing; Springer-Verlag, Berlin, 2009.

\bibitem{Z.Z}H.P. Zhang, S. Zhou, Characterizations for $\mathcal{P}_{\geqslant2}$-factor and $\mathcal{P}_{\geqslant3}$-factor covered graphs, Discrete Math. 309 (2009) 2067-2076.
\end{thebibliography}
\end{document}